\documentclass[3p]{elsarticle}

\usepackage{amsmath,amsthm,amssymb,mathrsfs}
\usepackage{enumerate}
\usepackage{slashed}
\usepackage{txfonts,dsfont}
\usepackage{aliascnt}
\usepackage[breaklinks,colorlinks]{hyperref}
\makeatletter
\newcommand{\autorefcheckize}[1]{%
  \expandafter\let\csname @@\string#1\endcsname#1%
  \expandafter\DeclareRobustCommand\csname relax\string#1\endcsname[1]{%
    \csname @@\string#1\endcsname{##1}\wrtusdrf{##1}}%
  \expandafter\let\expandafter#1\csname relax\string#1\endcsname
}
\makeatother

\theoremstyle{plain}

\newtheorem{theorem}{Theorem}[section]

\newaliascnt{lem}{theorem}
\newtheorem{lem}[lem]{Lemma}
 \aliascntresetthe{lem}

\newaliascnt{cor}{theorem}
\newtheorem{cor}[cor]{Corollary}
 \aliascntresetthe{cor}

\newaliascnt{prop}{theorem}

 \aliascntresetthe{prop}

\theoremstyle{remark}

\newtheorem{rem}{Remark}[section]

\newtheorem*{claim}{Claim}

\theoremstyle{definition}

\newtheorem{eg}{Example}[section]

\numberwithin{equation}{section}

\newcommand{\abs}[1]{\left\lvert#1\right\rvert}
\newcommand{\set}[1]{\left\{#1\right\}}
\newcommand{\hin}[2]{\left\langle#1,#2\right\rangle}

\newcommand*{\Rmn}[1]{\uppercase\expandafter{\romannumeral#1}}
\newcommand*{\To}{\longrightarrow}

\newcommand*{\dif}{\mathop{}\!\mathrm{d}}

\DeclareMathOperator{\Div}{div}

\journal{XXX}

\allowdisplaybreaks

\begin{document}

\begin{frontmatter}

\title{Mean curvature flow of surfaces in a hyperk\"ahler $4$-manifold \tnoteref{QS}}

\author[whu1,whu2]{Hongbing Qiu}
\ead{hbqiu@whu.edu.cn}

\author[whu1,whu2]{Linlin Sun \corref{sll1}}
\ead{sunll@whu.edu.cn}



\tnotetext[QS]{This work is partilly supported by NSFC (Nos. 11771339, 11801420), Fundamental Research Funds for the Central Universities (Nos. 2042019kf0198, 2042018kf0044) and the Youth Talent Training Program of Wuhan University. The first author would like to express his sincere gratitude to Professor Y. L. Xin who brought the related problem to the author, he also thanks Professor Tobias H. Colding for his invitation, to MIT for their hospitality. The authors thank Professors Qun Chen and Jixiang Fu for their suggestions and support. They also thank Dr. Hui Liu, Sheng Rao and Jun Sun for helpful discussions.}
\address[whu1]{School of Mathematics and Statistics, Wuhan University, Wuhan 430072, China}
\address[whu2]{Hubei Key Laboratory of Computational Science, Wuhan University, Wuhan, 430072, China
}

\cortext[sll1]{Corresponding author.}

\begin{abstract}
In this paper, we firstly prove that every hyper-Lagrangian submanifold $L^{2n} (n > 1)$ in a hyperk\"ahler $4n$-manifold  is a complex Lagrangian submanifold. Secondly, we 
demonstrate an optimal rigidity theorem with the condition on the complex phase map of self-shrinking surfaces in $\mathbb{R}^4$. Last but not least, by using the previous rigidity result, we show that the mean curvature flow from a closed surface with the image of the complex phase map contained in $\mathbb{S}^2\setminus \overline{\mathbb{S}}^{1}_{+}$ in a hyperk\"ahler $4$-manifold does not develop any Type \Rmn{1} singularity. 

\end{abstract}

\begin{keyword}
 hyper-Lagrangian\sep self-shrinker\sep rigidity\sep mean curvature flow\sep singularity

 \MSC[2020] 53E10, 53C24, 53C26

\end{keyword}

\end{frontmatter}


\section{Introduction}

Let $\widetilde{M}$ be a closed $m$-dimensional differential  manifold and  $(N, h)$ be an $\bar{n}$-dimensional Riemannian manifold which can be embedded into some Euclidean space. The mean curvature flow (MCF) in $N$ is a smooth one-parameter family of immersions $F_t = F(\cdot, t): {\widetilde{M}}^m \to N^{\bar{n}}$ with the corresponding image $\widetilde{M}_t = F_t(\widetilde{M})$ such that 
\begin{align}\label{ODE-add}
\begin{cases}
\frac{\partial}{\partial t}F(x,t)=\mathbf{H}(x,t), &(x,t)\in \widetilde{M}\times(0,T);\\
F(x,0)=F_0(x),&x\in \widetilde{M},
\end{cases}
\end{align}
is satisfied, where $\mathbf{H}(x, t)$ is the mean curvature vector of the isometric immersion  $\widetilde{M}_t$ in $N$ at $F(x, t)$ in $N^{\bar{n}}$. The  MCF \eqref{ODE-add} is a (degenerate) quasilinear parabolic evolution equation. By using the DeTurck's  trick (cf. \cite{DeT83}), one can prove that the MCF \eqref{ODE-add} has a smooth solution for short time interval $[0,T)$. Moreover, the maximum existence time $T$ satisfying (cf. \cite[Theorem 8.1]{Hui84})
\begin{align*}
    \limsup_{t\to T}\max_{\widetilde{M}_t}\abs{\mathbf{B}}=\infty,
\end{align*}
where $\mathbf{B}(x,t)$ is the second fundamental form of the isometric immersion $\widetilde{M}_t$ in $N$ at $F(x,t)$.
There are many significant works on MCF, see the references (not exhaustive): \cite{BreCho19, CheTia00, Eck95,Eck13,EckHui89,EckHui91, HanSun12,HasKle17,Hui84, Hui86, Hui90, Hui93, HuiSin99, HuiSin09, Ilm94,Ilm95, LeSes10, Smo12, SmoWan02, Wan02, Wan03, Wan11a, Whi05, Xin08} and the references therein.

Brakke \cite{Bra78} firstly studied the motion of a submanifold moving by its mean curvature from the viewpoint of geometric measure theory. In Huisken's seminal paper \cite{Hui84}, he showed that the closed convex hypersurfaces in Euclidean space $\mathbb{R}^{m+1} (m > 1)$ contracts to a single point under the MCF in finite time and the normalized  flow (area is fixed) converges to a sphere of the same area in infinite time. Later, Huisken \cite{Hui86} generalized his results to closed and uniformly convex hypersurfaces in a complete Riemannian manifold with bounded geometry. As time evolves, the MCF may develop singularities which can be classified as Type \Rmn{1} and Type \Rmn{2} according to the blow up rate of the second fundamental form with respect to time $t$. And Huisken \cite{Hui90} proved that after appropriate rescaling near the Type \Rmn{1} singularity the hypersurfaces converge to a self-similar solution of the MCF.  

 In the past twenty years, the MCF of higher codimension has made much progress. And symplectic MCF and Lagrangian MCF are two important class among them. Chen-Li \cite{CheLi01}  studied the symplectic MCF from a closed symplectic surface in a K\"ahler-Einstein $4$-manifold, by establishing a new monotonicity formula, and using blow up argument, they proved that the MCF has no Type \Rmn{1} singularity if the initial symplectic surface is closed in a  K\"ahler-Einstein surface with nonnegative scalar curvature. Almost at the same time, Wang \cite{Wan01} demonstrated the same conclusion by removing the condition on the curvature of the ambient manifold. Smoczyk \cite{Smo96} showed that the Lagrangian condition is preserved by the MCF when the ambient space is a Calabi-Yau $2n$-manifold (which is a closed $2n$-dimensional Riemannian manifold with holonomy contained in $\mathrm{SU}(n)$).  Afterwards, Wang \cite{Wan01} observed that almost calibrated Lagrangian submanifolds in a Calabi-Yau manifold can not develop Type \Rmn{1} singularities.  Chen-Li \cite{CheLi04} manifested that in this setting the tangent cone of the MCF at a singular point $(X_0, T)$ (here $T$ is the first blow up time of the MCF) is an integer rectifiable stationary Lagrangian varifold. Furthermore, Neves \cite{Nev07} studied finite time singularities for zero-Maslov class Lagrangian submanifolds in $\mathbb{C}^n$, a more general  condition than being almost calibrated. As a consequence, he showed that the Lagrangian MCF with zero-Maslov class does not develop any  Type \Rmn{1} singularity.  On the other hand, self-shrinkers are Type \Rmn{1} singularity models of the MCF, and there is a multitude  of excellent work on the classification and uniqueness problem for self-shrinkers (see e.g. \cite{AbrLan86, CaoLi13, ChaCheYua12, CheQiu16, ChePen15,CheOga16, CheWei15, ColMin12, Din18, DinXin14, DinXin14a, DinXinYan16, GuaZhu17, GuaZhu18, HuaWan11, LeSes11, Smo05, Son14, Wan11, Wan14, Wan16}). 

\vspace{2ex}

In this paper, we shall focus on the case where the ambient space is a hyperk\"ahler manifold.  A   hyperk\"ahler $4n$-manifold $M$  is a Riemannian manifold with holonomy contained in $\mathrm{Sp}(n)$.  It admits a 2-sphere-family of complex structures $J$ and the associated holomorphic symplectic form $\Omega_J \in \Omega^{2, 0}(M,J)$.  Leung-Wan \cite{LeuWan07} firstly introduced the concept of hyper-Lagrangian manifolds which is a generalization of complex Lagrangian submanifolds: A submanifold $L^{2n} \subset M^{4n}$ is called \emph{ hyper-Lagrangian} if each tangent space $T_xL \subset T_xM$ is a complex Lagrangian subspace with respect to  $\Omega_{J(x)}$ with varying $J(x) \in \mathbb{S}^2$. There are many restrictions for the hyper-Lagrangian submanifold, e.g., every   hyper-Lagrangian submanifold is a K\"ahler manifold with holomorphic normal bundle (cf. \cite[Corollary 4.2]{LeuWan07}). One can check that every  oriented surface immersed in a $4$-hyperk\"ahler manifold is automatically hyper-Lagrangian. Unfortunately, up to now, we do not know any nontrivial examples of hyper-Lagrangian submanifolds $L^{2n}$ in $M^{4n}$ for $n>1$.  Therefore, it is very important for us to  construct nontrivial examples of these submanifolds. Along this direction, we give the following restriction for the hyper-Lagrangian submanifold.

\begin{theorem}\label{thm:main1}
Every  hyper-Lagrangian submanifold $L^{2n} (n>1)$ in a hyperk\"ahler $4n$-manifold is a complex Lagrangian submanifold. 
\end{theorem}

 The authors \cite{LeuWan07} showed that the \emph{complex phase map} $J: L \to \mathbb{S}^2, x\mapsto J(x)$ satisfies the evolving harmonic map heat flow along the MCF and the hyper-Lagrangian condition  is preserved under the mean curvature flow. Moreover, they demonstrated that the MCF does not develop Type \Rmn{1} singularities if the image of $J$ of the initial closed hyper-Lagrangian submanifold is contained in an open hemisphere.  When $n=1$, their results are in accordance with \cite[Theorem 4.7]{CheLi01} and  \cite[Theorem A]{Wan01}. In addition, the method of the proof of  \cite[Theorem 5.1]{LeuWan07} could also be applied to almost calibrated Lagrangian submanifolds in a Calabi-Yau manifold and arrived at the same conclusion we mention previously. Recently, Kunikawa-Takahashi \cite{KunTak20} proved the longtime existence and convergence under the condition that the initial hyper-Lagrangian submanifold has sufficiently small twistor energy.   Due to \autoref{thm:main1}, it suffices to study surfaces in a hyperk\"ahler $4$-manifold. Notice that closed hyperk\"ahler $4$-manifolds are coincide with Calabi-Yau $4$-manifolds since $\mathrm{Sp}(1)=\mathrm{SU}(2)$.
 
 As we mention above, the problem of  singularities is  an extremely crucial topic in the MCF, we  are mainly interested in the geometry of surfaces in a hyperk\"ahler $4$-manifold and the corresponding mean curvature flow.  Note that in a hyperk\"ahler $4$-manifold, one can check that a surface being symplectic is equivalent to the condition that the image of the complex phase map is contained in an open hemisphere while a surface being Lagrangian is equivalent to the condition that the image of the complex phase map is contained in a great circle. Moreover, a Lagrangian  surface being almost calibrated is equivalent to the condition that the image of the complex phase map is contained in an open half great circle. Recall that  when  Jost-Xin-Yang \cite{JosXinYan12} studied the regularity of harmonic maps into spheres $\mathbb{S}^n$, they assumed that the image of harmonic maps is contained in  $\mathbb{S}^n\setminus \overline{\mathbb{S}}^{n-1}_{+}$ which is the maximal open convex supporting subset of $\mathbb{S}^n$.  Accordingly, it is natural to restrict the image of $J$ in $ \mathbb{S}^2\setminus \overline{\mathbb{S}}^{1}_{+}$ when we consider the MCF from a closed  surface in a hyperk\"ahler $4$-manifold, which can be regarded as a generalization of  both symplectic and almost calibrated Lagrangian MCF in a hyperk\"ahler $4$-manifold. In order to study the existence of the Type \Rmn{1} singularity of this MCF, we firstly study the geometry of the Type \Rmn{1} singularity, namely, the self-shrinking surface in $\mathbb{R}^4$, and we find that its complex phase map is a generalized harmonic map (cf. \cite{CheJosWan15}). Based on this observation, by using integral method, we obtain
 
 \begin{theorem}\label{thm:rigid}
 Let $X:\Sigma \to \mathbb{R}^4$ be a complete proper self-shrinking surface in $\mathbb{R}^4$. If the image of $J: \Sigma \to \mathbb{S}^2$ is contained in $\mathbb{S}^2\setminus \overline{\mathbb{S}}^{1}_{+}$, then $\Sigma$ must be a plane. 
 
\end{theorem}
 
 This theorem improves and generalizes the result of \cite{AreSun13}. We can give an example to illustrate that the above restriction on the image of $J$ is optimal (see \autoref{eg:cylinder} in section 5). Furthermore, we show that  if the image of the complex phase map of  the initial closed surface is contained in $\mathbb{S}^2\setminus \overline{\mathbb{S}}^{1}_{+}$, then the image of the complex phase map of the evolved surface is contained in some fixed compact subset of $\mathbb{S}^2\setminus \overline{\mathbb{S}}^{1}_{+}$ under the MCF.  Consequently, by using \autoref{thm:rigid} and applying the blow up analysis of MCFs, we prove the nonexistence of the Type \Rmn{1} singularity of the MCF.
 
 \begin{theorem}\label{thm:sing}
Let $\Sigma_0$ a closed surface immersed in hyperk\"ahler $4$-manifold $M$. Let $\Sigma_t \subset M (t\in [0, T) $ for some $T>0)$  be a family of  surfaces given by the mean curvature flow.  Suppose that  the image of the complex phase map $J: \Sigma_0 \to \mathbb{S}^2$
is contained in $\mathbb{S}^2\setminus \overline{\mathbb{S}}^{1}_{+}$, then the mean curvature flow has no Type \Rmn{1} singularity.
\end{theorem}

As a consequence, if the image of $J$ for the initial surface is contained in a great circle avoid a point, then the Lagrangian MCF has no Type \Rmn{1} singularity in a hyperk\"ahler $4$-manifold. The restriction on the image of $J$ is sharp in \autoref{thm:sing}, see \autoref{eg:calabi-tori} in section 5. 
When we consider that the ambient manifold is a  hyperk\"ahler $4$-manifold, \autoref{thm:sing} generalizes the corresponding results of \cite{LeuWan07, CheLi01, Wan01}. 

\vspace{2ex}

The article will be organized  as follows. We shall give some preliminaries in Section 2. In Section 3, we firstly give an equivalent condition of the hyper-Lagrangian,  from which it is easy to see that any surface in hyperk\"ahler 4-manifold is hyper-Lagrangian, then we prove that every hyper-Lagrangian submanifolds $L^{2n} (n > 1)$ in a hyperk\"ahler manifold $M^{4n}$ must be complex Lagrangian (\autoref{thm:main1}). Subsequently, we study the geometry of  surfaces in a hyperk\"ahler $4$-manifold in Section 4. Finally, in Section 5,  we demonstrate some rigidity theorems of self-shrinking surfaces and translating soliton surfaces in $\mathbb{R}^4$ (\autoref{thm:rigid}, \autoref{thm:soliton}), after that, we show that the MCF from a closed surface with the image of the complex phase map $J$ contained in $\mathbb{S}^2\setminus \overline{\mathbb{S}}^{1}_{+}$ does not develop any Type I singularity (\autoref{thm:sing}).


\section{Preliminaries}
In this section, we set some notations that will be used throughout the paper and recall some relevant definitions and results.

Let $M^{4n}$ be a $4n$-dimensional  hyperk\"ahler manifold, i.e., there exists two covariant constant anti-commutative almost complex structures $J_1, J_2$, i.e., $J_1, J_2$ are parallel with respect to the Levi-Civita connection and $J_1J_2=-J_2J_1$. Denote $J_3\coloneqq J_1J_2$, then the following quaternionic identities hold
\begin{align*}
J_1^2=J_2^2=J_3^2=J_1J_2J_3=-1. 
\end{align*}

Every $\mathrm{SO}(3)$ matrix preserves the quaternionic identities, i.e., $\set{\tilde J^{\alpha}\coloneqq\sum_{\beta=1}^3a_{\alpha\beta}J^{\beta}}$ satisfies the quaternionic identities
\begin{equation*}
\tilde J_1^2=\tilde J_2^2=\tilde J_3^2=\tilde J_1\tilde J_2\tilde J_3=-1.
\end{equation*}
In particular, for every unit vector $(a_1,a_2,a_3)\in\mathbb{R}^3$, we get a covariant constant almost complex structure $\sum_{\alpha=1}^3a_{\alpha}J_{\alpha}$, and this implies that $\left(M,\sum_{\alpha=1}^3a_{\alpha}J_{\alpha}\right)$ is a K\"ahler manifold.

\par

Let $\hat J=\sum_{\alpha=1}^3\lambda_{\alpha}J_{\alpha}$ be an almost complex structure on $M$. Let $\omega_{\hat J}$ be the K\"ahler form with respect to $\hat J$, then the associated symplectic $2$-form $\Omega_{\hat J}\in\Omega^{2,0}\left(M,\hat J\right)$ is given by
\begin{align*}
\Omega_{\hat J}=\omega_{K}+\sqrt{-1}\omega_{K\hat J},
\end{align*}
where $K=\sum_{\alpha=1}^3\mu_{\alpha}J_{\alpha}$ is an almost complex structure which is orthogonal to $\hat J$ in the sense that $\sum_{\alpha=1}^3\lambda_{\alpha}\mu_{\alpha}=0$. If $\hat J$ is parallel, then $\Omega_{\hat J}$ is holomorphic with respect to the covraiant constant almost complex structure $\hat J$.


Let $\omega_{\alpha}$ be the K\"ahler form associated with the almost complex structure $J_{\alpha}$, then $\left(M,J_1\right)$ is a K\"ahler manifold and 
\begin{align*}
\Omega_{J_1}=\omega_2+\sqrt{-1}\omega_3\in H^{2,0}\left(M,J_1\right)
\end{align*}
is the associated holomorphic symplectic $2$-form. We say that a submanifold $L^{2n}$ of $M^{4n}$ is complex Lagrangian if for some covariant constant complex structure $\hat J$ of $M$ such that the associated holomorphic symplectic $2$-form $\Omega_{\hat J}$ vanished everywhere on $L$. Without loss of generality, assume $\hat J=J_1$, then $L$ is a K\"ahler submanifold of the K\"ahler manifold $(M,J_1)$. In particular, $L$ is a minimal submanifold of $M$. Moreover, both $L\subset\left(M, J_2\right)$ and $L\subset\left(M, J_3\right)$ are  Lagrangian immersions.

We say that $L^{2n}$ is a hyper-Lagrangian submanifold of $M^{4n}$ if there is an almost complex structure $\hat J=\sum_{\alpha=1}^3\lambda_{\alpha}J_{\alpha}$ such that the associated symplectic $2$-form $\Omega_{\hat J}$ vanished everywhere on $L$.  The map 
\begin{align*}
J:L\To\mathbb{S}^2,\quad x\mapsto J(x)\coloneqq(\lambda_1,\lambda_2,\lambda_3)
\end{align*}
is called the complex phase map. In other words, $L$ is hyper-Lagrangian iff each $T_xL$ is a  complex Lagrangian subspace of $T_xM$. Here we say that $T_xL$ is a complex Lagrangian subspace of $T_xM$ if for some complex structure $\hat J=\sum_{\alpha=1}^3\lambda_{\alpha}J_{\alpha}$ we have
\begin{align*}
\bar g\left(K\cdot,\cdot\right)\vert_{T_xL}=0,
\end{align*}
for all almost complex structures $K=\sum_{\alpha=1}^3\mu_{\alpha}J_{\alpha}$ which are orthogonal to $\hat  J$. Therefore, $L$ is complex Lagrangian iff $L$ is hyper-Lagrangian with  constant complex phase map.

The complex phase map $J$ defines an almost complex structure $\tilde J=\sum_{\alpha=1}^3\lambda_{\alpha}J_{\alpha}\vert_{TL\To TL}$ on $L$ and an almost complex structure $\tilde J^{\bot}=\sum_{\alpha=1}^3\lambda_{\alpha}J_{\alpha}\vert_{T^{\bot}L\To T^{\bot}L}$ on $T^{\bot}L$. Denoted $\bar\nabla, \nabla$ and $\nabla^{\bot}$ by the Levi-Civita connections on $TM,TL$ and $T^{\bot}L$ respectively. Denoted $\bar R, R$ and $R^{\bot}$ by the Riemannian curvatures on $TM,TL$ and $T^{\bot}L$ respectively. For $V\in \Gamma(TL)$, let $\Delta_V:= \Delta + \left< V, \nabla \cdot \right>$, where $\Delta$ is the usual Laplacian operator with respect to $\nabla$.

\section{Every hyper-Lagrangian submanifold but surface  is complex Lagrangian}

In this section, we shall give a proof of \autoref{thm:main1}. In particular, every hyper-Lagrangian submanifold $L^{2n}$ in a hyperk\"ahler manifold $M^{4n}$ is minimal when $n>1$. 

First, we have the following
\begin{lem}\label{lem:hyper-Lagrangian}
$L$ is hyper-Lagrangian iff
\begin{align*}
J_{\alpha}\vert_{TL\To TL}=\lambda_{\alpha}\tilde J,\quad\alpha=1,2,3,
\end{align*}
iff
\begin{align*}
J_{\alpha}\vert_{T^{\bot}L\To T^{\bot}L}=\lambda_{\alpha}\tilde J^{\bot},\quad\alpha=1,2,3.
\end{align*}
\end{lem}
\begin{proof}
Under the orthogonal decomposition $T_xM=T_xL\oplus T_x^{\bot}L$, we write
\begin{align*}
J_{\alpha}=\begin{pmatrix}A_{\alpha}&B_{\alpha}\\
-B_{\alpha}^T&C_{\alpha}
\end{pmatrix},\quad 
\alpha=1,2,3.
\end{align*}
Let $A=(a_{\alpha\beta})_{1\leq \alpha,\beta\leq 3}\in\mathrm{SO}(3)$ where $\lambda_{\alpha}=a_{1\alpha}$. Set $\tilde J_{\beta}=\sum_{\alpha=1}^3a_{\beta\alpha}J_{\alpha}$. Since $L$ is hyper-Lagrangian, we get
\begin{align*}
\tilde J_{2}=\begin{pmatrix}0&\sum_{\alpha=1}^3a_{2\alpha}B_{\alpha}\\
-\sum_{\alpha=1}^3a_{2\alpha}B_{\alpha}^{T}&0
\end{pmatrix},\quad \tilde J_{3}=\begin{pmatrix}0&\sum_{\alpha=1}^3a_{3\alpha}B_{\alpha}\\
-\sum_{\alpha=1}^3a_{3\alpha}B_{\alpha}^{T}&0
\end{pmatrix},
\end{align*}
or equivalently,
\begin{align*}
A_{\alpha}=\lambda_{\alpha}\tilde J,\quad\alpha=1,2,3,
\end{align*}
which is also equivalent to 
\begin{align*}
C_{\alpha}=\lambda_{\alpha}\tilde J^{\bot},\quad\alpha=1,2,3.
\end{align*}
\end{proof}
\begin{rem}This Lemma claims that every surface immersed in a hyperk\"ahler $4$-manifold is automatically hyper-Lagrangian.
\end{rem}

The first restriction of hyper-Lagrangian submanifolds is the following (\cite[Corollary 4.2]{LeuWan07}).
\begin{lem}\label{lem:restriction1}
If $L$ is hyper-Lagrangian, then $\left(L,\tilde J\right)$ is a K\"ahler manifold with a holomorphic normal bundle.
\end{lem}
\begin{proof}
We will give an alternative proof here. 
By \autoref{lem:hyper-Lagrangian}, for all $X\in\Gamma(TL)$, we have
\begin{align*}
\lambda_{\alpha}\tilde JX=\left(J_{\alpha}X\right)^{\top},\quad\alpha=1,2,3.
\end{align*}
For all $Y\in\Gamma(TL)$, we get
\begin{align*}
&Y(\lambda_{\alpha})\tilde JX+\lambda_{\alpha}\left(\nabla_Y\tilde J\right)X+\lambda_{\alpha}\tilde J\nabla_YX
= Y(\lambda_{\alpha})\tilde JX+\lambda_{\alpha}\nabla_Y(\tilde JX)
= \nabla_Y(\lambda_{\alpha}\tilde JX)\\
=&\nabla_Y\left(J_{\alpha}X\right)^{\top}
=\bar\nabla_Y\left(J_{\alpha}X\right)^{\top}-\mathbf{B}\left(Y,\left(J_{\alpha}X\right)^{\top}\right)
=\bar\nabla_Y\left(\sum_{j=1}^{2n}\hin{J_{\alpha}X}{e_j}e_j\right)-\lambda_{\alpha}\mathbf{B}\left(Y,\tilde JX\right)\\
=&\left(\bar\nabla_Y(J_{\alpha}X)\right)^{\top}+\sum_{j=1}^{2n}\hin{J_{\alpha}X}{\mathbf{B}(Y,e_j)}e_j+\sum_{j=1}^{2n}\hin{J_{\alpha}X}{e_j}\mathbf{B}(Y,e_j)-\lambda_{\alpha}\mathbf{B}\left(Y,\tilde JX\right)\\
=&\left(J_{\alpha}\nabla_YX+J_{\alpha}\mathbf{B}(Y,X)\right)^{\top}+\mathbf{A}^{\left(J_{\alpha}X\right)^{\bot}}(Y),
\end{align*}
where $\set{e_j}_{1\leq j\leq 2n}$ is a local orthonormal frame of $TL$ and $\mathbf{A}$ is the shape operator. Thus,
\begin{align*}
Y(\lambda_{\alpha})\tilde JX+\lambda_{\alpha}\left(\nabla_Y\tilde J\right)X=\left(J_{\alpha}\mathbf{B}(Y,X)\right)^{\top}+\mathbf{A}^{\left(J_{\alpha}X\right)^{\bot}}(Y).
\end{align*}
Since $\sum_{\alpha=1}^3\lambda_{\alpha}^2=1$, we get
\begin{align*}
\left(\nabla_Y\tilde J\right)X=\left(\tilde J_1\mathbf{B}(Y,X)\right)^{\top}+\mathbf{A}^{\left(\tilde J_1X\right)^{\bot}}(Y)=0.
\end{align*}
Therefore, $\nabla\tilde J=0$ which implies that $\left(L,\tilde J\right)$ is a K\"ahler manifold. Similarly, one can prove that $\nabla^{\bot}\tilde J^{\bot}=0$.
\end{proof}

\begin{lem}If $L$ is hyper-Lagrangian, then
\begin{align}\label{eq:0}
\hin{\mathbf{B}(X,Y)}{J_{\alpha}Z}=\hin{\mathbf{B}(X,Z)}{J_{\alpha}Y}-X\left(\lambda_{\alpha}\right)\hin{\tilde JY}{Z},\quad\forall X, Y, Z\in TL.
\end{align}
or equivalently
\begin{align}\label{eq:1}
Y(\lambda_{\alpha})\tilde JX=\left(J_{\alpha}\mathbf{B}(Y,X)\right)^{\top}+\mathbf{A}^{\left(J_{\alpha}X\right)^{\bot}}(Y),\quad\forall X, Y\in TL.
\end{align}
Moreover,
\begin{align}\label{eq:2}
\mathbf{B}\left(X,\tilde JY\right)=\tilde J^{\bot}\mathbf{B}\left(X,Y\right)+\sum_{\alpha=1}^3X\left(\lambda_{\alpha}\right)J_{\alpha}Y.
\end{align}
\end{lem}
\begin{proof} Without loss of generality, we may assume $\nabla X=\nabla Y=\nabla Z=0$ at a considered point. We shall compute at this considered point,
\begin{align*}
\hin{\mathbf{B}(X,Y)}{J_{\alpha}Z}=&\hin{\bar\nabla_XY}{J_{\alpha}Z}\\
=&\bar\nabla_X\hin{Y}{J_{\alpha}Z}-\hin{Y}{J_{\alpha}\bar\nabla_XZ}\\
=&X\left(\lambda_{\alpha}\hin{Y}{\tilde JZ}
\right)-\hin{Y}{J_{\alpha}\mathbf{B}(X,Z)}\\
=&X\left(\lambda_{\alpha}\right)\hin{Y}{\tilde JZ}+\hin{J_{\alpha}Y}{\mathbf{B}(X,Z)}.
\end{align*}
Here the last two equalities followed from  \autoref{lem:hyper-Lagrangian} and \autoref{lem:restriction1} respectively.

For the second claim, we compute
\begin{align*}
\mathbf{B}\left(X,\tilde JY\right)=&\left(\bar\nabla_X\left(\tilde JY\right)\right)^{\bot}\\
=&\sum_{\alpha=1}^3\left(\bar\nabla_X\left(\lambda_{\alpha}J_{\alpha}Y\right)\right)^{\bot}\\
=&\left[\sum_{\alpha=1}^3X\left(\lambda_{\alpha}\right)J_{\alpha}Y+\sum_{\alpha=1}^3\lambda_{\alpha}J_{\alpha}\mathbf{B}\left(X,Y\right)+\sum_{\alpha=1}^3\lambda_{\alpha}J_{\alpha}\nabla_XY\right]^{\bot}\\
=&\sum_{\alpha=1}^3X\left(\lambda_{\alpha}\right)J_{\alpha}Y+\mathbf{B}\left(X,Y\right).
\end{align*}
\end{proof}

Now we can give the following
\begin{proof}[Proof of \autoref{thm:main1}]
For each $\alpha\in\set{1,2,3}$, on one hand, according to \eqref{eq:1}, we have
\begin{align*}
2nY(\lambda_{\alpha})=&\sum_{j=1}^{2n}\hin{\left(J_{\alpha}\mathbf{B}(Y,e_j)\right)^{\top}}{\tilde Je_j}+\sum_{j=1}^{2n}\hin{\mathbf{A}^{\left(J_{\alpha}e_j\right)^{\bot}}(Y)}{\tilde Je_j}\\
=&\sum_{j=1}^{2n}\hin{J_{\alpha}\mathbf{B}(Y,e_j)}{\tilde Je_j}+\sum_{j=1}^{2n}\hin{\mathbf{B}(Y,\tilde Je_j)}{J_{\alpha}e_j}\\
=&2\sum_{j=1}^{2n}\hin{\mathbf{B}(Y,\tilde Je_j)}{J_{\alpha}e_j}\\
=&2\sum_{j=1}^{2n}\hin{\mathbf{A}^{\left(J_{\alpha}e_j\right)^{\bot}}\left(\tilde Je_j\right)}{Y}.
\end{align*}
Thus
\begin{align}\label{eq:1-n}
n\nabla\lambda_{\alpha}=\sum_{j=1}^{2n}\mathbf{A}^{\left(J_{\alpha}e_j\right)^{\bot}}\left(\tilde Je_j\right).
\end{align}
On the other hand, from \eqref{eq:0}, we derive
\begin{align*}
\sum_{j=1}^{2n}\hin{\mathbf{B}\left(\tilde Je_j,X\right)}{J_{\alpha}e_j}=&\sum_{j=1}^{2n}\hin{\mathbf{B}\left(\tilde Je_j,e_j\right)}{J_{\alpha}X}+\sum_{j=1}^{2n}\hin{\nabla\lambda_{\alpha}}{\tilde Je_j}\hin{X}{\tilde Je_j}\\
=&\hin{\nabla\lambda_{\alpha}}{X}.
\end{align*}
Thus,
\begin{align}\label{eq:1-1}
\nabla\lambda_{\alpha}=\sum_{j=1}^{2n}\mathbf{A}^{\left(J_{\alpha}e_j\right)^{\bot}}\left(\tilde Je_j\right).
\end{align}
Combining \eqref{eq:1-n} with \eqref{eq:1-1}, we conclude that
\begin{align*}
(n-1)\nabla\lambda_{\alpha}=0,\quad\alpha=1,2,3.
\end{align*}
Consequently, if $n>1$, then $\dif J=0$ which implies that the complex phase map $J$ is a constant map. In particular, $L$ is complex Lagrangian when $n>1$.
\end{proof}

\section{Surfaces }
Let $\Sigma$ be a closed surface immersed in a hyperk\"ahler $4$-manifold $M$. As shown in the previous section, we know that $\Sigma$ is a hyper-Lagrangian surface in $M$ with holomorphic norm bundle $T^{\bot}\Sigma$.  Leung and Wan obtained the following formula (\cite[equation (1.1)]{LeuWan07})
\begin{align}\label{eq:LW}
    \partial J=\dfrac{\sqrt{-1}}{2}\iota_{\mathbf{H}}\Omega_{J}.
\end{align}
Introduce the curvature form $H\in\Gamma\left(T^{*}\Sigma\otimes J^{-1}T\mathbb{S}^2\right)$ as follows:
\begin{align*}
    H(X)\coloneqq\left(\hin{\mathbf{H}}{J_1X},\hin{\mathbf{H}}{J_2X},\hin{\mathbf{H}}{J_3X}\right)\in T_{J(x)}\mathbb{S}^2,\quad\forall X\in T_x\Sigma.
\end{align*}

Recall the complex structure $J_{\mathbb{S}^2}$ on $T\mathbb{S}^2$: for every tangent vector field $(a,b,c)\in T\mathbb{S}^2\subset T\mathbb{R}^3$,
\begin{align*}
J_{\mathbb{S}^2}(a,b,c)=&\left(\lambda_1,\lambda_2,\lambda_3\right)\times(a,b,c)\\
=&\left(\begin{vmatrix}\lambda_2&\lambda_3\\
b&c
\end{vmatrix},\begin{vmatrix}\lambda_3&\lambda_1\\
c&a
\end{vmatrix},\begin{vmatrix}\lambda_1&\lambda_2\\
a&b
\end{vmatrix}\right).
\end{align*}

 We can reformulate \eqref{eq:LW} as follows
\begin{lem}\label{lem:LW}
\begin{align}\label{eq:LW1}
\partial J=&-\dfrac{\sqrt{-1}}{4}H-\dfrac{1}{4}J_{\mathbb{S}^2}\circ H.
\end{align}
Consequently,
\begin{align*}
    \abs{\partial J}^2=\dfrac14\abs{\mathbf{H}}^2.
\end{align*}
In particular, $\Sigma$ is minimal iff $J$ is anti-holomorphic.
\end{lem}

\begin{proof}

For every tangent vector $X\in T\Sigma$ and normal vector $V\in T^{\bot}\Sigma$, we have
\begin{align*}
\hin{V}{J_{1}\tilde JX}=&\lambda_2\hin{V}{J_3X}-\lambda_3\hin{V}{J_2X},\\
\hin{V}{J_{2}\tilde JX}=&-\lambda_1\hin{V}{J_3X}+\lambda_3\hin{V}{J_1X},\\
\hin{V}{J_{3}\tilde JX}=&\lambda_1\hin{V}{J_2X}-\lambda_2\hin{V}{J_1X}.
\end{align*}
In other words,
\begin{align}\label{eq:3-1}
\left(\hin{V}{J_1\tilde JX},\hin{V}{J_2\tilde JX},\hin{V}{J_3\tilde JX}\right)=J_{\mathbb{S}^2}\left(\hin{V}{J_1X},\hin{V}{J_2X},\hin{V}{J_3X}\right).
\end{align}
Thus,
\begin{align*}
    H\circ\tilde J=J_{\mathbb{S}^2}\circ H.
\end{align*}
According to \eqref{eq:2}, we get
\begin{align*}
    \mathbf{H}=\sum_{\alpha=1}^3\tilde J^{\bot}J_{\alpha}\nabla\lambda_{\alpha}.
\end{align*}
It follows that
\begin{align*}
    \hin{\mathbf{H}}{J_{\alpha}X}=&\hin{\sum_{\beta=1}^3\tilde J^{\bot}J_{\beta}\nabla\lambda_{\beta}}{J_{\alpha}X}
    =-\hin{\sum_{\beta=1}^3J_{\beta}\nabla\lambda_{\beta}}{\tilde J^{\bot}J_{\alpha}X}
    =\hin{\sum_{\beta=1}^3J_{\beta}\nabla\lambda_{\beta}}{J_{\alpha}\tilde JX}.
\end{align*}
Combining with \eqref{eq:3-1}, we get
\begin{align*}
    H\left(X\right)=&J_{\mathbb{S}^2}\left(\hin{\sum_{\beta=1}^3J_{\beta}\nabla\lambda_{\beta}}{J_{1}X},\hin{\sum_{\beta=1}^3J_{\beta}\nabla\lambda_{\beta}}{J_{2}X},\hin{\sum_{\beta=1}^3J_{\beta}\nabla\lambda_{\beta}}{J_{3}X}\right).
\end{align*}
Direct calculation yields
\begin{align*}
    \hin{\sum_{\beta=1}^3J_{\beta}\nabla\lambda_{\beta}}{J_{1}X}=&\hin{\nabla\lambda_1+\lambda_2\tilde J\nabla\lambda_3-\lambda_3\tilde J\nabla\lambda_2}{X},\\
    \hin{\sum_{\beta=1}^3J_{\beta}\nabla\lambda_{\beta}}{J_{2}X}=&\hin{\nabla\lambda_2+\lambda_3\tilde J\nabla\lambda_1-\lambda_1\tilde J\nabla\lambda_3}{X},\\
    \hin{\sum_{\beta=1}^3J_{\beta}\nabla\lambda_{\beta}}{J_{3}X}=&\hin{\nabla\lambda_3+\lambda_1\tilde J\nabla\lambda_2-\lambda_2\tilde J\nabla\lambda_1}{X},
\end{align*}
which implies
\begin{align*}
    \left(\hin{\sum_{\beta=1}^3J_{\beta}\nabla\lambda_{\beta}}{J_{1}X},\hin{\sum_{\beta=1}^3J_{\beta}\nabla\lambda_{\beta}}{J_{2}X},\hin{\sum_{\beta=1}^3J_{\beta}\nabla\lambda_{\beta}}{J_{3}X}\right)=&\dif J\left(X\right)-J_{\mathbb{S}^2}\dif J\left(\tilde JX\right).
\end{align*}
Therefore, we have
\begin{align*}
    H=\dif J\circ\tilde J+J_{\mathbb{S}^2}\circ\dif J.
\end{align*}
By the definition
\begin{align*}
\partial J=&\dfrac14\left(\dif J-J_{\mathbb{S}^2}\circ\dif J\circ\tilde J\right)-\dfrac{\sqrt{-1}}{4}\left(\dif J\circ\tilde J+J_{\mathbb{S}^2}\circ\dif J\right).
\end{align*}
Hence we obtain 
\begin{align*}
\partial J=&-\dfrac{\sqrt{-1}}{4}H-\dfrac{1}{4}J_{\mathbb{S}^2}\circ H.
\end{align*}
\end{proof}

\begin{theorem}
Let $\Sigma$ be a closed surface immersed in a hyperk\"ahler $4$-manifold $M$, then
\begin{align}\label{eq:tension}
\tau\left(J\right)=&J_{\mathbb{S}^2}\left(\Div\left(J_{1}\mathbf{H}\right)^{\top}, \Div\left(J_{2}\mathbf{H}\right)^{\top}, \Div\left(J_{3}\mathbf{H}\right)^{\top}\right).
\end{align}
Moreover,
\begin{align*}
\det\left(\dif J\right)=&\kappa+\kappa^{\bot},
\end{align*}
where
\begin{align*}
\kappa=R(e_1,e_2,e_1,e_2),\quad\kappa^{\bot}=\hin{R^{\bot}\left(e_1,e_2\right)\nu_2}{\nu_1}.
\end{align*}
Here $e_1,e_2,\nu_1,\nu_2$ determines the orientation of $M$. As a consequence,
\begin{align*}
2\deg(J)=\chi\left(T\Sigma\right)+\chi\left(T^{\bot}\Sigma\right).
\end{align*}
\end{theorem}
\begin{proof}

Since
\begin{align*}
\partial J=&\dfrac14\left(\dif J-J_{\mathbb{S}^2}\circ\dif J\circ\tilde J\right)-\dfrac{\sqrt{-1}}{4}\left(\dif J\circ\tilde J+J_{\mathbb{S}^2}\circ\dif J\right),
\end{align*}
we have
\begin{align*}
\sum_{j=1}^2\left(\nabla_{e_j}\partial J\right)\left(e_j\right)=\dfrac14\left(1-\sqrt{-1}J_{\mathbb{S}^2}\right)\tau\left(J\right).
\end{align*}
Moreover,
\begin{align*}
\Div\left(J_{\alpha}\mathbf{H}\right)^{\top}=&\sum_{j=1}^2\hin{\nabla_{e_j}\left(J_{\alpha}\mathbf{H}\right)^{\top}}{e_j}\\
=&\sum_{j=1}^2e_j\hin{J_{\alpha}\mathbf{H}}{e_j}-\sum_{j=1}^2\hin{J_{\alpha}\mathbf{H}}{\nabla_{e_j}e_j}\\
=&\sum_{j=1}^2\hin{J_{\alpha}\nabla_{e_j}^{\bot}\mathbf{H}}{e_j}-\sum_{j=1}^2\hin{J_{\alpha}\mathbf{A}^{\mathbf{H}}\left(e_j\right)}{e_j}+\sum_{j=1}^2\hin{J_{\alpha}\mathbf{H}}{\mathbf{H}}\\
=&\sum_{j=1}^2\hin{J_{\alpha}\nabla_{e_j}^{\bot}\mathbf{H}}{e_j}.
\end{align*}
Thus,
\begin{align}\label{eq:divJH}
    \Div\left(J_{\alpha}\mathbf{H}\right)^{\top}=\sum_{j=1}^2\hin{J_{\alpha}\nabla_{e_j}^{\bot}\mathbf{H}}{e_j}.
\end{align}
From \eqref{eq:LW1}, we get
\begin{align*}
\tau\left(J\right)=&J_{\mathbb{S}^2}\left(\Div\left(J_{1}\mathbf{H}\right)^{\top}, \Div\left(J_{2}\mathbf{H}\right)^{\top}, \Div\left(J_{3}\mathbf{H}\right)^{\top}\right).
\end{align*}
According to \autoref{lem:LW}, 
\begin{align}
\det\left(\dif J\right)=&\abs{\partial J}^2-\abs{\bar\partial J}^2\notag\\
=&2\abs{\partial J}^2-\dfrac12\abs{\dif J}^2\notag\\
=&\dfrac12\abs{\mathbf{H}}^2-\dfrac12\abs{\dif J}^2.\label{eq:detJ-1}
\end{align}
By using \eqref{eq:2}, we have
\begin{align*}
    \abs{\dif J(X)}^2\abs{Y}^2=\abs{\mathbf{B}\left(X,\tilde JY\right)-\tilde J^{\bot}\mathbf{B}\left(X,Y\right)}^2.
\end{align*}
It follows that
\begin{align*}
\abs{\dif J}^2=&\dfrac12\sum_{j,k=1}^2\abs{\mathbf{B}\left(e_j,\tilde Je_k\right)-\tilde J^{\bot}\mathbf{B}\left(e_j,e_k\right)}^2\\
=&\dfrac{1}{2}\abs{\mathbf{B}\left(e_j,\tilde Je_k\right)}^2+\dfrac{1}{2}\abs{\tilde J^{\bot}\mathbf{B}\left(e_j,e_k\right)}^2-\sum_{j,k=1}^2\hin{\mathbf{B}\left(e_j,\tilde Je_k\right)}{\tilde J^{\bot}\mathbf{B}\left(e_j,e_k\right)}\\
=&\abs{\mathbf{B}}^2-\sum_{j,k=1}^2\hin{\mathbf{B}\left(e_j,\tilde Je_k\right)}{\tilde J^{\bot}\mathbf{B}\left(e_j,e_k\right)}.
\end{align*}
Set $e_2=\tilde Je_1=\tilde J_1e_1, \nu_1=\tilde J_2e_1$ and $\nu_2=\tilde J^{\bot}\nu_1=\tilde J_3e_1$, then by applying the Gauss equation and Ricci equation, we obtain
\begin{align*}
\abs{\dif J}^2-\abs{\mathbf{H}}^2=&\abs{\mathbf{B}}^2-\abs{\mathbf{H}}^2-\sum_{\alpha=1}^2\sum_{j,k=1}^{2}\hin{\mathbf{B}\left(e_j,\tilde Je_k\right)}{\nu_{\alpha}}\hin{\nu_{\alpha}}{\tilde J^{\bot}\mathbf{B}(e_j,e_k)}\\
=&\abs{\mathbf{B}}^2-\abs{\mathbf{H}}^2\\
&+\sum_{j,k=1}^{2}\hin{\mathbf{B}\left(e_j,\tilde Je_k\right)}{\nu_{1}}\hin{\mathbf{B}(e_j,e_k)}{\nu_2}-\sum_{j,k=1}^{2}\hin{\mathbf{B}\left(e_j,\tilde Je_k\right)}{\nu_{2}}\hin{\mathbf{B}(e_j,e_k)}{\nu_1}\\
=&\abs{\mathbf{B}}^2-\abs{\mathbf{H}}^2+\sum_{j=1}^2\hin{R^{\bot}\left(\tilde Je_j,e_j\right)\nu_2}{\nu_1}-\sum_{j=1}^2\bar R\left(\tilde Je_j,e_j,\nu_1,\nu_2\right)\\
=&\abs{\mathbf{B}}^2-\abs{\mathbf{H}}^2-2\hin{R^{\bot}\left(e_1,e_2\right)\nu_2}{\nu_1}-2\bar R\left(e_1,e_2,\tilde J_2e_1,\tilde J_2e_2\right)\\
=&\abs{\mathbf{B}}^2-\abs{\mathbf{H}}^2-2\hin{R^{\bot}\left(e_1,e_2\right)\nu_2}{\nu_1}-2\bar R\left(e_1,e_2,e_1,e_2\right)\\
=&-2R(e_1,e_2,e_1,e_2)-2\hin{R^{\bot}\left(e_1,e_2\right)\nu_2}{\nu_1}.
\end{align*}
Namely,
\begin{align}\label{eq:detJ-2}
\abs{\dif J}^2-\abs{\mathbf{H}}^2=&-2\kappa-2\kappa^{\bot}.
\end{align}
Substituting \eqref{eq:detJ-2} into \eqref{eq:detJ-1}, we derive
\begin{align*}
\det\left(\dif J\right)=\kappa+\kappa^{\bot}.
\end{align*}
By applying the Gauss-Bonnet formula, we get
\begin{align*}
4\pi\deg(J)=\int_{\Sigma}\det\left(\dif J\right)=\int_{\Sigma}\kappa+\int_{\Sigma}\kappa^{\bot}=2\pi\chi\left(T\Sigma\right)+2\pi\chi\left(T^{\bot}\Sigma\right),
\end{align*}
i.e.,
\begin{align*}
2\deg(J)=\chi\left(T\Sigma\right)+\chi\left(T^{\bot}\Sigma\right).
\end{align*}
\end{proof}

\begin{cor}Let $\Sigma$ be a closed  surface immersed in a hyperk\"ahler $4$-manifold $M$. Assume
\begin{align*}
\Div\left(\left(J_{\alpha}\mathbf{H}\right)^{\top}\right)=0,\quad\alpha=1,2,3,
\end{align*}
and $2\chi\left(T\Sigma\right)+\chi\left(T^{\bot}\Sigma\right)>0$, then the complex phase map $J$ is holomorphic.
\end{cor}
\begin{proof}According to \eqref{eq:tension}, the assumption means that $J$ is a harmonic map. Then this Corollary is a consequence of the following observation: if $J$ is not holomorphic, then 
\begin{align*}
\Delta\log\abs{\bar\partial J}=&\det\left(\dif J\right)+\kappa\\
=&2\kappa+\kappa^{\bot},
\end{align*}
holds when $\bar\partial J\neq0$. 
\end{proof}

Moreover,
\begin{cor}Let $\Sigma$ be a closed  surface immersed in a hyperk\"ahler $4$-manifold $M$. Assume
\begin{align*}
\Div\left(\left(J_{\alpha}\mathbf{H}\right)^{\top}\right)=0,\quad\alpha=1,2,3,
\end{align*}
and $\chi\left(T^{\bot}\Sigma\right)<0$, then $\Sigma$ is minimal.
\end{cor}

\begin{proof}It is a consequence of \autoref{lem:LW} and the following observation: if $J$ is not anti-holomorphic, then 
\begin{align*}
\Delta\log\abs{\partial J}=&-\det\left(\dif J\right)+\kappa\\
=&-\kappa^{\bot},
\end{align*}
holds when $\partial J\neq0$. 
\end{proof}

\section{Nonexistence of  Type \texorpdfstring{\Rmn{1}}{I} singularity of MCF} 

In this section, we consider the mean curvature flow from a closed  surface $\Sigma$ in a hyperk\"ahler $4$-manifold $M$, i.e, we consider
\begin{align}\label{eqn-MCF}
\begin{cases}
\dfrac{\partial F}{\partial t}=\mathbf{H},&\Sigma\times[0,T);\\
F\left(\cdot,t\right)=F_0\left(\cdot\right),&\Sigma.
\end{cases}
\end{align}
Here $F_0:\Sigma\To M$ is an  isometric immersion. This flow blows up when
\begin{align*}
\limsup_{t\to T}\max_{\Sigma_t}\abs{\mathbf{B}}=\infty.
\end{align*}
We say that the mean curvature flow $F$ has Type \Rmn{1} singularity at $T>0$ if
\begin{align*}
\limsup_{t\to T}\sqrt{T-t}\max_{\Sigma_t}\abs{\mathbf{B}}\leq C,
\end{align*}
for some positive constant $C$.

\vskip12pt

We shall need the following theorem which is owed to Leung-Wan (see \cite[Theorem 3.4]{LeuWan07}), here we would like to give an alternative proof.

\begin{theorem}[\cite{LeuWan07}]\label{thm: harmonic-heat}  
The complex phase maps of the mean curvature flow \eqref{eqn-MCF} $J: \Sigma_t \To \mathbb{S}^2$ form an evolving harmonic map heat flow, i.e., 
\begin{align*}
\dfrac{\partial J}{\partial t} = \tau(J),
\end{align*}
where $\tau(J)$ is the tension field of $J$ with respect to the induced metric $g_t$ on $\Sigma_t$.
\end{theorem}

\begin{proof} Let $\set{e_i}$ be a local orthonormal evolving frame field on $\Sigma_t$, then both  $[e_i,\partial_t]$ and  $[\tilde Je_i,\partial_t]$ are local tangent vector fields and
\begin{align*}
\lambda_{\alpha}\delta_{ij}=\hin{J_{\alpha}e_i}{\tilde Je_j}.
\end{align*}
Differentiate with respect to t on both sides of the above equality, 
\begin{align*}
&\dfrac{\partial\lambda_{\alpha}}{
\partial t}\delta_{ij}=\dfrac{\partial}{\partial t}\hin{J_{\alpha}e_i}{\tilde Je_j}
=\hin{J_{\alpha}\bar\nabla_{\partial_t}e_i}{\tilde Je_j}+\hin{J_{\alpha}e_i}{\bar\nabla_{\partial_t}\tilde Je_j}\\
=&\lambda_{\alpha}\hin{\tilde J\left(\bar\nabla_{\partial_t}e_i\right)^{\top}}{\tilde Je_j}+\hin{J_{\alpha}\left(\bar\nabla_{\partial_t}e_i\right)^{\bot}}{\tilde Je_j}
+\lambda_{\alpha}\hin{\tilde Je_i}{\left(\bar\nabla_{\partial_t}\tilde Je_j\right)^{\top}}+\hin{J_{\alpha}e_i}{\left(\bar\nabla_{\partial_t}\tilde Je_j\right)^{\bot}}\\
=&\lambda_{\alpha}\hin{\bar\nabla_{\partial_t}e_i}{\tilde Je_j}+\hin{J_{\alpha}\left(\bar\nabla_{e_i}\partial_t\right)^{\bot}}{\tilde Je_j}
+\lambda_{\alpha}\hin{\tilde Je_i}{\left(\bar\nabla_{\partial_t}\tilde Je_j\right)^{\top}}+\hin{J_{\alpha}e_i}{\left(\bar\nabla_{\tilde Je_j}\partial_t\right)^{\bot}}.
\end{align*}
Since $\sum_{\alpha=1}^3\lambda_{\alpha}^2=1$, we get
\begin{align*}
0=&\hin{\bar\nabla_{\partial_t}e_i}{\tilde Je_j}+\hin{\tilde Je_i}{\left(\bar\nabla_{\partial_t}\tilde Je_j\right)^{\top}}.
\end{align*}
Thus
\begin{align*}
\dfrac{\partial\lambda_{\alpha}}{
\partial t}\delta_{ij}=&\hin{J_{\alpha}\left(\bar\nabla_{e_i}\partial_t\right)^{\bot}}{\tilde Je_j}+\hin{J_{\alpha}e_i}{\left(\bar\nabla_{\tilde Je_j}\partial_t\right)^{\bot}}.
\end{align*}
Consequently,
\begin{align*}
2\dfrac{\partial\lambda_{\alpha}}{\partial t}=&\sum_{j=1}^{2}\hin{J_{\alpha}\left(\bar\nabla_{e_j}\partial_t\right)^{\bot}}{\tilde Je_j}+\sum_{j=1}^{2}\hin{J_{\alpha}e_j}{\left(\bar\nabla_{\tilde Je_j}\partial_t\right)^{\bot}}
=2\sum_{j=1}^{2}\hin{J_{\alpha}e_j}{\left(\bar\nabla_{\tilde Je_j}\partial_t\right)^{\bot}}.
\end{align*}
By (\ref{eq:divJH}) and (\ref{eq:3-1}), we get
\[
(\tau(J))^\alpha = - \sum_{j=1}^2 \hin{ \nabla^{\bot}_{e_j} \mathbf{H}}{J_\alpha \tilde{J} e_j} = -  \sum_{j=1}^2 \hin{ \nabla^{\bot}_{\tilde{J}e_j} \mathbf{H}}{-J_\alpha  e_j } =  \sum_{j=1}^2 \hin{ \nabla^{\bot}_{\tilde{J}e_j} \mathbf{H}}{ J_\alpha  e_j}
\]
Hence
\begin{align*}
\dfrac{\partial \lambda_\alpha}{\partial t}-(\tau(J))^\alpha =\sum_{j=1}^2\hin{\left(\bar\nabla_{\tilde Je_j}\left(\frac{\partial F}{\partial t} - \mathbf{H}\right) \right)^{\bot}}{J_{\alpha}e_j} = 0, \quad  \alpha = 1, 2, 3.
\end{align*}
Namely, 
\[
\frac{\partial J}{\partial t} = \tau(J).
\]
\end{proof}

Next, we show that the complex phase map is a generalized harmonic map (cf. \cite{CheJosWan15})

\begin{theorem}\label{thm:self-shrinker}Let $X:\Sigma\To\mathbb{R}^4$ be a self-shrinker, i.e., $\mathbf{H}=-\frac12X^{\bot}$, then the complex phase map $J:\Sigma\To\mathbb{S}^2$ satisfies
\begin{align*}
\tau\left(J\right)=\dfrac12\dif J\left(X^{\top}\right).
\end{align*}

\end{theorem}
\begin{proof}
Since
\begin{align*}
\bar\nabla_{e_j}X^{\bot}=&-\mathbf{A}^{X^{\bot}}\left(e_j\right)-\mathbf{B}\left(e_j,X^{\top}\right),
\end{align*}
we get
\begin{align*}
\sum_{j=1}^2\hin{\nabla^{\bot}_{\tilde Je_j}\mathbf{H}}{J_{\alpha}e_j}=&-\dfrac12\sum_{j=1}^2\hin{\nabla^{\bot}_{\tilde Je_j}X^{\bot}}{J_{\alpha}e_j}\\
=&-\dfrac12\sum_{j=1}^2\hin{\bar\nabla_{\tilde Je_j}X^{\bot}}{J_{\alpha}e_j}-\dfrac12\sum_{j=1}^2\hin{\mathbf{A}^{X^{\bot}}\left(\tilde Je_j\right)}{J_{\alpha}e_j}\\
=&\dfrac12\sum_{j=1}^2\hin{\mathbf{A}^{X^{\bot}}\left(\tilde Je_j\right)+\mathbf{B}\left(\tilde Je_j,X^{\top}\right)}{J_{\alpha}e_j}-\dfrac12\sum_{j=1}^2\hin{\mathbf{A}^{X^{\bot}}\left(\tilde Je_j\right)}{J_{\alpha}e_j}\\
=&\dfrac12\sum_{j=1}^2\hin{\mathbf{A}^{\left(J_{\alpha}e_j\right)^{\bot}}\left(\tilde Je_j\right)}{X}
=\dfrac12\hin{\nabla\lambda_{\alpha}}{X}.
\end{align*}
The last equality follows from \eqref{eq:1-1}. Applying \eqref{eq:divJH}, we conclude
\begin{align*}
\tau(J)=&J_{\mathbb{S}^2}\left(\Div\left(J_{1}\mathbf{H}\right)^{\top}, \Div\left(J_{2}\mathbf{H}\right)^{\top}, \Div\left(J_{3}\mathbf{H}\right)^{\top}\right)\\
=&J_{\mathbb{S}^2}\left(\sum_{j=1}^2\hin{J_1\nabla_{e_j}^{\bot}\mathbf{H}}{e_j}, \sum_{j=1}^2\hin{J_2\nabla_{e_j}^{\bot}\mathbf{H}}{e_j},\sum_{j=1}^2\hin{J_3\nabla_{e_j}^{\bot}\mathbf{H}}{e_j}\right)\\
=&-J_{\mathbb{S}^2}\left(\sum_{j=1}^2\hin{\nabla_{\tilde Je_j}^{\bot}\mathbf{H}}{J_1\tilde Je_j}, \sum_{j=1}^2\hin{\nabla_{\tilde Je_j}^{\bot}\mathbf{H}}{J_2\tilde Je_j},\sum_{j=1}^2\hin{\nabla_{\tilde Je_j}^{\bot}\mathbf{H}}{J_3\tilde Je_j}\right)\\
=&\left(\sum_{j=1}^2\hin{\nabla_{\tilde Je_j}^{\bot}\mathbf{H}}{J_1e_j}, \sum_{j=1}^2\hin{\nabla_{\tilde Je_j}^{\bot}\mathbf{H}}{J_2e_j},\sum_{j=1}^2\hin{\nabla_{\tilde Je_j}^{\bot}\mathbf{H}}{J_3e_j}\right)\\
=&\dfrac12\dif J\left(X^{\top}\right).
\end{align*}

\end{proof}

Using \autoref{thm:self-shrinker} and integral method, we shall prove \autoref{thm:rigid}.

Denote $\overline{\mathbb{S}}^1_{+}\coloneqq\set{\left. (x_1, x_2, x_3) \in \mathbb{S}^2\right| x_1 = 0, x_2\geq 0 }$ and put $\mathbb{V}\coloneqq\mathbb{S}^2 \setminus \overline{\mathbb{S}}^1_{+} $.

\begin{proof}[Proof of \autoref{thm:rigid}]
Firstly, we shall prove that when the image of the complex phase map $J$ is contained in an open hemisphere, then $\Sigma $ must be a plane. 

Indeed,  according to \autoref{thm:self-shrinker}, we know that 
\begin{align*}
\tau\left(J\right)=\dfrac12\dif J\left(X^{\top}\right).
\end{align*}
Let $\rho$ be the distance function on $\mathbb{S}^2$, and define $\psi\coloneqq(1-\cos \rho)$. 
Let $ u\coloneqq\psi \circ J $. Then
\begin{align*}
\Delta_{-\frac{X^{\top}}{2}} u = \Delta u - \frac{1}{2}\hin{ X^{\top}}{\nabla u} = \sum_{j=1}^2\mathrm{Hess(\psi)}\left(\dif J\left(e_j\right),\dif J\left(e_j\right)\right) + d\varphi\left(\tau(J) + dJ\left(-\frac{X^{\top}}{2}\right)\right)= (\cos\rho )\abs{dJ}^2 \geq 0.
\end{align*}
with the equality holds iff $\dif J=0$. Since 
\begin{align*}
\Delta_{-\frac{X^{\top}}{2}} u = \Delta u - \frac{1}{2}\hin{ X^{\top}}{\nabla u} = e^{\frac{\abs{X}^2}{4}}{\rm div}\left(e^{-\frac{\abs{X}^2}{4}}\nabla u\right).
\end{align*}
It follows that
\begin{align*}
\Div\left(e^{-\frac{\abs{X}^2}{4}}\nabla u\right)\geq0.
\end{align*}
For every compactly supported Lipschitz function $\eta$ on $\mathbb{R}^4$, since $\Sigma$ is proper, we know that $\eta\vert_{\Sigma}$ is also compactly supported in $\Sigma$. Consequently,
\begin{align*}
\int_{\Sigma}\Div\left(e^{-\frac{\abs{X}^2}{4}}\nabla u\right)\eta^2 u\geq0.
\end{align*}
 which implies that
 \begin{align*}
 \int_{\Sigma}\abs{\nabla u}^2\eta^2e^{-\frac{\abs{X}^2}{4}}+2\int_{\Sigma}\hin{\nabla u}{\nabla\eta}\eta ue^{-\frac{\abs{X}^2}{4}}\leq0.
 \end{align*}
 Notice that
 \begin{align*}
 \abs{\nabla\left(\eta u\right)}^2=\abs{\nabla\eta}^2u^2+\eta^2\abs{\nabla u}^2+2\hin{\nabla u}{\nabla\eta}\eta u,
 \end{align*}
 we obtain
 \begin{align*}
 \int_{\Sigma}\abs{\nabla\left(\eta u\right)}^2e^{-\frac{\abs{X}^2}{4}}\leq\int_{\Sigma}\abs{\nabla\eta}^2u^2e^{-\frac{\abs{X}^2}{4}}.
 \end{align*}
For every $R>0$, choose
 \begin{align*}
 \eta(y)=\begin{cases}
 1,&\abs{y}\leq R;\\
 \dfrac{2R-\abs{y}}{R},&R<\abs{y}<2R;\\
 0,&\abs{y}\geq 2R.
 \end{cases}
 \end{align*}
 Then $\nabla\eta=0$ for $\abs{x}<R$ or $\abs{x}>2R$. For $R<\abs{x}<2R$,
 \begin{align*}
 \abs{\nabla\eta(x)}\leq\dfrac{1}{R}.
 \end{align*}
 Consequently,
 \begin{align*}
 \int_{\Sigma\cap B_R}\abs{\nabla u}^2e^{-\frac{\abs{X}^2}{4}}\leq\dfrac{1}{R^2}\int_{\Sigma\cap\left(B_{2R}\setminus B_{R}\right)}u^2e^{-\frac{\abs{X}^2}{4}}\leq\dfrac{1}{R^2}\int_{\Sigma\cap\left(B_{2R}\setminus B_{R}\right)}e^{-\frac{\abs{X}^2}{4}}.
 \end{align*}
 Since $\Sigma$ is proper, we have
 \begin{align*}
 \int_{\Sigma}e^{-\frac{\abs{X}^2}{4}}<\infty.
 \end{align*}
 Letting $R\to\infty$, we get
 \begin{align*}
 \int_{\Sigma}\abs{\nabla u}^2e^{-\frac{\abs{X}^2}{4}}=0.
 \end{align*}
 This implies that $u \equiv constant$, namely  $J \equiv constant$. Hence $\Sigma$ is minimal and $X^{\bot}=0$. The fact $X=X^{\top}$ gives $\mathbf{B}\left(X,\cdot\right)=0$. By the minimal condition, we know that $\Sigma$ is totally geodesic. Since $\Sigma$ is complete, we conclude that $\Sigma$ is a plane.
 
 Secondly, we consider the projection $\pi$ from $\mathbb{S}^2$ onto $\overline{\mathbb{D}}^2$ (here $\overline{\mathbb{D}}^2$ is a 2-dimensional closed unit disk)
 
 \[  
 \pi: \mathbb{S}^2 \To \overline{\mathbb{D}}^2,\quad (x_1, x_2, x_3)\to (x_1, x_2).
 \]
 Then $x \in \mathbb{V}$ if and only if $\pi(x)$ is contained in the domain obtained by removing the radius connecting $(0, 0)$ and $(0, 1)$ from the closed unit disk. Therefore for any $x \in \mathbb{V}$, there exists a unique $(0, 1]-$valued function $r$ and a unique $(0, 2\pi)-$valued function $\varphi$ on $\mathbb{V}$, such that 
 \[
 \pi(x) = (r\sin\varphi, r\cos\varphi).
 \]
 Direct computation gives us (see  \cite[formula (2.12)]{JosXinYan12})
 \[
 {\rm Hess}\varphi = - r^{-1} (d\varphi \otimes \dif r + \dif r \otimes d\varphi).
 \]
 It follows that  
 \begin{align*}
 \Delta_{-\frac{X^{\top}}{2}}(\varphi\circ J) = {\rm Hess}\varphi (dJ(e_i), dJ(e_i)) = -2 (r\circ J)^{-1} \left< \nabla(r\circ J), \nabla (\varphi \circ J) \right>. 
 \end{align*}
 Thus 
 \begin{align*}
& \Div\left((r\circ J)^2e^{-\frac{\abs{X}^2}{4}}\nabla (\varphi\circ J)\right) 
  =  (r\circ J)^2 \Div\left( e^{-\frac{\abs{X}^2}{4}}\nabla(\varphi\circ J) \right) + \hin{ \nabla (r\circ J)^2}{ e^{-\frac{\abs{X}^2}{4}}\nabla(\varphi \circ J)} \\
  =& e^{-\frac{\abs{X}^2}{4}}(r\circ J)^2  \Delta_{-\frac{X^{\top}}{2}}(\varphi\circ J) + \hin{\nabla (r\circ J)^2}{ e^{-\frac{\abs{X}^2}{4}}\nabla(\varphi \circ J)}  \\
 =& -2 \hin{(r\circ J) \nabla(r\circ J)}{e^{-\frac{\abs{X}^2}{4}}\nabla (\varphi \circ J)} + \hin{ \nabla(r\circ J)^2}{ e^{-\frac{\abs{X}^2}{4}}\nabla(\varphi\circ J)} = 0.
 \end{align*}
 
 Let $\eta$ be as before, multiplying $\eta^{2}  \cdot (\varphi\circ J) $ with both sides of the above equality, 
  \begin{align*}
  0= & \int_{\Sigma} \eta^2 \cdot (\varphi\circ J) \Div \left( (r\circ J)^2 e^{-\frac{\abs{X}^2}{4}} \nabla(\varphi \circ J) \right) \\
  =& \int_\Sigma \Div \left( \eta^2 \cdot (\varphi\circ J) (r\circ J)^2 e^{-\frac{\abs{X}^2}{4}} \nabla(\varphi \circ J)  \right) - \int_\Sigma \hin{ \nabla (\eta^2 \cdot (\varphi \circ J))}{\nabla (\varphi\circ J)}(r\circ J)^2 e^{-\frac{\abs{X}^2}{4}}\\
  =& -\int_\Sigma \eta^2 \abs{\nabla (\varphi\circ J)}^2 ( r\circ J )^2 e^{-\frac{\abs{X}^2}{4}} - 2\int_\Sigma \hin{(\varphi\circ J) \nabla \eta}{ \eta \nabla(\varphi\circ J)} (r\circ J)^2 e^{-\frac{\abs{X}^2}{4}} \\
  \leq & -\int_{\Sigma} \eta^2 \abs{\nabla (\varphi\circ J)}^2 ( r\circ J )^2 e^{-\frac{\abs{X}^2}{4}}  + 2 \int_\Sigma (\varphi \circ J)^2\abs{\nabla \eta}^2 (r\circ J)^2 e^{-\frac{\abs{X}^2}{4}}\\
  &+ \frac{1}{2} \int_\Sigma  \eta^2 \abs{\nabla (\varphi\circ J)}^2 ( r\circ J )^2 e^{-\frac{\abs{X}^2}{4}}. 
  \end{align*}
 Therefore we obtain
 \begin{align*}
 \int_\Sigma \eta^2 \abs{\nabla (\varphi\circ J)}^2 ( r\circ J )^2 e^{-\frac{\abs{X}^2}{4}} \leq 4 \int_\Sigma (\varphi \circ J)^2|\nabla \eta|^2 (r\circ J)^2 e^{-\frac{\abs{X}^2}{4}}.
 \end{align*}
 It follows that
 \begin{align*}
 \int_{\Sigma \cap B_R} \abs{\nabla(\varphi \circ J)}^2 (r\circ J)^2 e^{-\frac{\abs{X}^2}{4}} \leq & \int_\Sigma \eta^2 |\nabla (\varphi\circ J)|^2 (r\circ J)^2 e^{-\frac{\abs{X}^2}{4}}\\
  \leq & 4 \int_\Sigma (\varphi\circ J)^2 \abs{\nabla \eta}^2 (r\circ J)^2 e^{-\frac{\abs{X}^2}{4}}  
 \leq \frac{16 \pi^2}{R^2} \int_{\Sigma \cap (B_{2R}\backslash B_R)} e^{-\frac{\abs{X}^2}{4}}.
 \end{align*}
 Letting $R\to \infty$, then we derive $\abs{\nabla (\varphi\circ J)} \equiv 0$. Thus $ \varphi\circ J \equiv \varphi_0 \in (0, 2\pi)$.
 
 Denote $b_0\coloneqq(\sin \varphi_0, \cos \varphi_0, 0)$, note that for any $p\in \Sigma$, $J(p) = ((r\circ J(p)) \sin(\varphi\circ J(p)), (r\circ J(p)) \cos(\varphi\circ J(p)), x_3 )$, then for any $p \in \Sigma$, the inner product of $J(p)$ and $b_0$ in $\mathbb{R}^3$ is $\langle J(p), b_0 \rangle = r(J(p))(\sin^2\varphi_0 + \cos^2\varphi_0) = r(J(p))>0$. This implies that the image of $J$ is contained in an open hemisphere centered at $b_0$. 
 Hence $\Sigma$ is a plane.
 \end{proof}

\begin{eg}[Cylinder]\label{eg:cylinder}
Consider the cylinder 
\begin{align*}
    \Sigma^2\coloneqq\set{ (x^1, x^2, x^3, x^4) \in \mathbb{R}^4 | (x^1)^2 + (x^2)^2 = 1, x^4 = 0 }\subset\mathbb{R}^4,
\end{align*}
which is a nontrivial self-shrinker. It is easy to see that $\nu\coloneqq x^1 \frac{\partial}{\partial x^1} + x^2 \frac{\partial }{\partial x^2}$ and $\frac{\partial}{\partial x^4}$ are normal vectors of $\Sigma$ in $\mathbb{R}^4$ and $e_1\coloneqq -x^2 \frac{\partial}{\partial x^1} + x^1 \frac{\partial}{\partial x^2}, e_2\coloneqq \frac{\partial}{\partial x^3}$ are tangent vectors of $\Sigma$. Let $\tilde{J}$ be the almost complex  structure on $\Sigma$ with $\tilde{J}e_1 = e_2, \tilde{J}e_2 = - e_1$. 

In $\mathbb{R}^4$, under the natural basis $\set{ \frac{\partial}{\partial x^1}, \frac{\partial}{\partial x^2}, \frac{\partial}{\partial x^3}, \frac{\partial}{\partial x^4}}$, we have
\begin{align*}
J_1=\begin{pmatrix}
         0 & -1 & 0 & 0 \\
         1 & 0 & 0 & 0 \\
         0 & 0 & 0 & -1 \\
         0 & 0 & 1 & 0
        \end{pmatrix},\quad
       J_2=\begin{pmatrix}
         0 & 0 & -1 & 0 \\
         0 & 0 & 0 & 1 \\
         1 & 0 & 0 & 0 \\
         0 & -1 & 0 & 0
        \end{pmatrix}, \quad
       J_3=\begin{pmatrix}
         0 & 0 & 0 & -1 \\
         0 & 0 & -1 & 0 \\
         0 & 1 & 0 & 0 \\
         1 & 0 & 0 & 0
        \end{pmatrix}.
           \end{align*}
Direct computation gives us 
\begin{align*}
J_1 e_1 =& x^2 \frac{\partial}{\partial x^2}  +  x^1 \frac{\partial}{\partial x^1} = \nu, \quad J_1 e_2 = - \frac{\partial}{\partial x^4},  \\
J_2 e_1 =& x^2 \frac{\partial}{\partial x^3} + x^1 \frac{\partial}{\partial x^4}, \quad \quad J_2 e_2 = \frac{\partial}{\partial x^1}, \\
J_3 e_1 =& x^2 \frac{\partial}{\partial x^4} -x^1 \frac{\partial}{\partial x^3}, \quad J_3e_2 = \frac{\partial}{\partial x^2}.
\end{align*}
Therefore we have
\begin{align*}
\left.J_1\right|_{\Sigma}= &0,  \\
\left.J_2\right|_{\Sigma} e_1 =& x^2 e_2 = x^2 \tilde{J}e_1, \quad \left.J_2\right|_{\Sigma} e_2= \frac{\partial}{\partial x^1} - \left< \frac{\partial}{\partial x^1}, \nu \right>\nu = -x^2 e_1 =x^2 \tilde{J} e_2, \\
\left.J_3\right|_{\Sigma} e_1 =&  -x^1 e_2 =-x^1 \tilde{J}e_1, \quad \left.J_3\right|_{\Sigma} e_2 = \frac{\partial}{\partial x^2} - \left< \frac{\partial}{\partial x^2}, \nu \right>\nu = x^1 e_1 = -x^1 \tilde{J}e_2.
\end{align*}
Thus the complex phase map $J$ can be represented by $(0, x^2, -x^1)$. Note that $ (x^1)^2 + (x^2)^2 = 1$, this implies the image of $J$ is a great circle. Clearly, even we add a point to $\mathbb{V}$, it will contain a great circle. Hence this example illustrates that the image restriction of the complex phase map in  Theorem \ref{thm:rigid} is optimal.
\end{eg}

\begin{cor}

Let $X: \Sigma^2 \to \mathbb{R}^4$ be a complete proper symplectic self-shrinking surface, then $\Sigma$ must be a  plane.

\end{cor}

\begin{proof}
The symplectic condition implies that the image of the complex phase map is contained in an open hemisphere.
\end{proof}

\begin{rem}
Arezzo-Sun \cite{AreSun13} proved that  complete proper symplectic self-shrinker surfaces in $\mathbb{R}^4$ must be  a plane under different conditions on the second fundamental form, flat normal bundle or bounded geometry (see \cite[Main Theorems 3, 4, 5]{AreSun13}).
\end{rem}

For the rigidity of translating soliton, we have the following

\begin{theorem}\label{thm:soliton}

Let $X: \Sigma^2 \to \mathbb{R}^4$ be a complete translating soliton  surface with flat normal bundle. Assume the image of the complex phase map is contained in a regular ball in $\mathbb{S}^2$, i.e., a geodesic ball $B_R(q)$ disjoint from the cut locus of $q$ and $R < \frac{\pi}{2}$, then $\Sigma$ has to be a plane.

\end{theorem}

\begin{proof}

Since we can view $\Sigma$ as a hyper-Lagrangian submanifold in $\mathbb{R}^4$ with respect to some almost complex structure $J$,   Let $\set{ e_1, e_2=\tilde{J}e_1 }$ be a local orthonormal frame field on $\Sigma$  such that $\nabla e_i =0$ at the considered point. Denote $\nu_1=\tilde{J}_2 e_1, \nu_2 =\tilde J^{\bot}\nu_1= \tilde{J_1}\tilde J_2e_2 =\tilde{J}_3 e_1$, then $\set{\nu_1, \nu_2}$  is a local orthonormal frame  normal field along $\Sigma$.  Recall the translating soliton equation $\mathbf{H}=-V_{0}^{\bot}$, here $V_0$ is a fixed unit vector. Denote $V\coloneqq V_{0}^{\top}$,  we obtain
\begin{equation*}\aligned
L_{-V} e_i =& \overline{\nabla}_{-V}e_i - \overline{\nabla}_{e_i}(-V) = -\hin{V}{e_j} \overline{\nabla}_{e_j}e_i +\overline{\nabla}_{e_i}\left(\hin{V_0}{ e_j }e_j\right)   \\
=& -\hin{V_0}{ e_j}\mathbf{B}\left(e_j, e_i\right)  + \hin{V_0}{ \mathbf{B}\left(e_j, e_i\right)} e_j + \hin{V_0}{e_j}\mathbf{B}\left(e_i, e_j\right)\\
=&\hin{-\mathbf{H}}{\mathbf{B}\left(e_i,  e_j\right)}e_j =  - H^\alpha h^{\alpha}_{ij}e_j.
\endaligned
\end{equation*}
It follows that
 \begin{equation*}\aligned
 -\dfrac{1}{2}\left(L_{-V} g\right)(e_i) =& -\dfrac{1}{2}\left(L_{-V} g\right)\left(e_i, e_j\right)e_j = \dfrac{1}{2}g\left(L_{-V}e_i, e_j\right)e_j + \dfrac{1}{2}g\left(L_{-V}e_j, e_i\right)e_j
  =  - H^{\alpha}h^{\alpha}_{ij}e_j.
  \endaligned
 \end{equation*}
 The Gauss equation and the above equality imply that
\begin{equation*}\aligned
{\rm Ric}_{-V} (e_i)= {\rm Ric}(e_i) -  \frac{1}{2}(L_{-V} g)(e_i)
=(H^\alpha h^{\alpha}_{ij}- h^{\alpha}_{ik}h^{\alpha}_{jk})e_j  - H^\alpha h^{\alpha}_{ij}e_j = -  \sum_{\alpha,j, k}h^{\alpha}_{ik}h^{\alpha}_{jk}e_j.
\endaligned
 \end{equation*}
 From \eqref{eq:2}, we obtain
 \begin{align*}
     \hin{\dif J\left(e_i\right)}{\dif J\left(e_j\right)}=&\hin{\mathbf{B}\left(e_i,\tilde Je_1\right)-\tilde J^{\bot}\mathbf{B}\left(e_i,e_1\right)}{\mathbf{B}\left(e_j,\tilde Je_1\right)-\tilde J^{\bot}\mathbf{B}\left(e_j,e_1\right)}\\
     =&\hin{\mathbf{B}\left(e_i,e_2\right)-\tilde J^{\bot}\mathbf{B}\left(e_i,e_1\right)}{\nu_1}\hin{\mathbf{B}\left(e_j,e_2\right)-\tilde J^{\bot}\mathbf{B}\left(e_j,e_1\right)}{\nu_1}\\
     &+\hin{\mathbf{B}\left(e_i,e_2\right)-\tilde J^{\bot}\mathbf{B}\left(e_i,e_1\right)}{\nu_2}\hin{\mathbf{B}\left(e_j,e_2\right)-\tilde J^{\bot}\mathbf{B}\left(e_j,e_1\right)}{\nu_2}\\
     =&\left(h^1_{i2}+h^2_{i1}\right)\left(h^1_{j2}+h^2_{j1}\right)+\left(h^2_{i2}-h^1_{i1}\right)\left(h^2_{j2}-h^1_{j1}\right)\\
     =&\sum_{\alpha=1}^2\sum_{k=1}^2h^{\alpha}_{ik}h^{\alpha}_{jk}+h^{1}_{i2}h^{2}_{j1}+h^2_{i1}h^{1}_{j2}-h^{2}_{i2}h^1_{j1}-h^{1}_{i1}h^2_{j2}.
 \end{align*}
 We conclude that
 \begin{align*}
    &\hin{\dif J\left(e_i\right)}{\dif J\left(e_j\right)}\left(\hin{\dif J\left(e_i\right)}{\dif J\left(e_j\right)}-\sum_{\alpha=1}^2\sum_{k=1}^2h^{\alpha}_{ik}h^{\alpha}_{jk}\right) \\
    =&\hin{\dif J\left(e_i\right)}{\dif J\left(e_j\right)}\left(h^{1}_{i2}h^{2}_{j1}+h^2_{i1}h^{1}_{j2}-h^{2}_{i2}h^1_{j1}-h^{1}_{i1}h^2_{j2}\right)\\
    =&2\hin{\dif J\left(h^1_{2i}e_i\right)}{\dif J\left(h^2_{1j}e_j\right)}-\hin{\dif J\left(h^2_{2i}e_i\right)}{\dif J\left(h^{1}_{1j}e_j\right)}\\
    =&2\hin{\dif J\left(\mathbf{A}^1\left(e_2\right)\right)}{\dif J\left(\mathbf{A}^2\left(e_1\right)\right)}-2\hin{\dif J\left(\mathbf{A}^1\left(e_1\right)\right)}{\dif J\left(\mathbf{A}^2\left(e_2\right)\right)}.
 \end{align*}

Since the normal bundle is flat, by the Ricci equation,  the coefficients of the second fundamental form $h^\alpha_{ij}$ satisfy 
\begin{align*}
     \sum_{i=1}^2\left(h^\alpha_{ij}h^\beta_{ik} - h^\beta_{ij}h^\alpha_{ik}\right) = 0,
\end{align*}
 which means that two $(2\times 2)$ matrices 
 \[
 (h^1_{ij}), (h^2_{ij})
 \]
 can be diagonalized simultaneously at a fixed point. 
 
 Therefore for any $p\in \Sigma$, we can choose a local frame field $\{e_1, e_2\}$ around $p$ such that $h^\alpha_{ij} = \Lambda^\alpha_{i}\delta_{ij}$ at $p$, i.e., 
 \begin{align*}
     \mathbf{A}^{\alpha}\left(e_i\right)=\Lambda^{\alpha}e_i,\quad i,\alpha=1,2.
 \end{align*}
 Hence at the $p$,
 \begin{align*}
     \hin{\dif J\left(e_1\right)}{\dif J\left(e_2\right)}=&\Lambda^1_{2}\Lambda^2_1-\Lambda^1_{1}\Lambda^2_2.
 \end{align*}
 Thus,
 \begin{align*}
     &\hin{\dif J\left(e_i\right)}{\dif J\left(e_j\right)}\left(\hin{\dif J\left(e_i\right)}{\dif J\left(e_j\right)}-\sum_{\alpha=1}^2\sum_{k=1}^2h^{\alpha}_{ik}h^{\alpha}_{jk}\right) \\
     =&2\left(\Lambda^1_{2}\Lambda^2_1-\Lambda^1_{1}\Lambda^2_2\right)\hin{\dif J\left(e_1\right)}{\dif J\left(e_2\right)}\\
     \geq&0.
 \end{align*}
 
It is easy to see that  the curvature tensor of $\mathbb{S}^2$ satisfies
\[
\sum_{i, j}R^{\mathbb{S}^2}(dJ(e_i), dJ(e_j), dJ(e_i), dJ(e_j)) = \abs{dJ}^4 - \sum_{i, j}\hin{\dif J\left(e_i\right)}{\dif J\left(e_j\right)}^2.
\]
From the translator equation, we get
\begin{equation*}\aligned
&\hin{\nabla^{\bot}_{\tilde{J}e_j}\mathbf{H}}{ J_\alpha e_j} = - \hin{ \nabla^{\bot}_{\tilde{J}e_j} V_{0}^{\bot}}{J_\alpha e_j}
= -\hin{\overline{\nabla}_{\tilde{J}e_j} V_0^{\bot}}{J_\alpha e_j}-  \hin{ \mathbf{A}^{V_0^{\bot}}(\tilde{J}e_j)}{ J_\alpha e_j } \\
=& - \hin{\overline{\nabla}_{\tilde{J}e_j} V_0}{J_\alpha e_j} +  \hin{ \overline{\nabla}_{\tilde{J}e_j} V_0^{\top}}{ J_\alpha e_j} -  \lambda_\alpha \hin{\mathbf{H}}{ V_0^{\bot}}
=  \hin{\overline{\nabla}_{\tilde{J}e_j} V_0^{\top}}{ J_\alpha e_j} -  \lambda_\alpha \hin{\mathbf{H}}{ V_0}.
\endaligned
\end{equation*}
Since 
\begin{equation*}\aligned
 \hin{\overline{\nabla}_{\tilde{J}e_j} V_0^{\top}}{ J_\alpha e_j} = &  \hin{V_0}{ \mathbf{B}(\tilde{J}e_j, e_k)} \hin{ e_k}{ J_\alpha e_j}+ \hin{V_0}{e_k}\hin{\mathbf{ B}(\tilde{J}e_j, e_k)}{ J_\alpha e_j}  \\
=&   \lambda_\alpha \hin{V_0}{\mathbf{H}} + \hin{\mathbf{A}^{(J_{\alpha}e_j)^\bot}(\tilde{J}e_j)}{V_{0}^{\top}}.
\endaligned
\end{equation*}
Therefore
\begin{align*}
    (\tau(J))^\alpha = \hin{\nabla^{\bot}_{\tilde{J}e_j}\mathbf{H}}{ J_\alpha e_j} = \hin{\mathbf{A}^{(J_{\alpha}e_j)^\bot}(\tilde{J}e_j)}{V_{0}^{\top}}= \hin{\nabla \lambda_\alpha}{ V_0^{\top}}= \hin{\nabla \lambda_\alpha}{ V}.
\end{align*}
Namely, 
\[
\tau(J) + dJ(-V) =0.
\]
Thus $J$ is a $-V$-harmonic map, then the Bochner formula (see  \cite[Lemma 1]{CheJosQiu12}) gives us
\begin{align*}
\frac{1}{2}\Delta_{-V} |dJ|^2 =& \abs{\nabla dJ}^2 + \sum_i \hin{ dJ({\rm Ric}_{-V}(e_i))}{dJ(e_i)} - \sum_{i, j}R^{\mathbb{S}^2}(dJ(e_i), dJ(e_j), dJ(e_i), dJ(e_j)) \\
= &\abs{\nabla dJ}^2 - \sum_{i, j, \alpha, k} h^\alpha_{ik} h^\alpha_{jk} \hin{dJ(e_i)}{ dJ(e_j) } -(|dJ|^4 - \sum_{i, j}\hin{ dJ(e_i)}{dJ(e_j)}^2 \\
\geq &\abs{\nabla dJ}^2-\abs{dJ}^4.
 \end{align*}
 Let $\rho $ be the distance function on $\mathbb{S}^2$, and $h$ the Riemannian metric of $\mathbb{S}^2$. Define $\psi=1-\cos\rho$,  then ${\rm Hess}(\psi) =( \cos\rho)h$. 

Since for any $X=(x_1, ..., x_{4}) \in \mathbb{R}^{4}$, let $r=|X|$, then we have 
\begin{align*}
\nabla r^2 =&  2X^{\top},  \quad |\nabla r| \leq 1 \\
\Delta r^2 = & 4 + 2\hin{\mathbf{H}}{ X } \leq 4+2r.
\end{align*}

Since $J(\Sigma) \subset B_R(q)\subset \mathbb{S}^2$, note that $R<\frac{\pi}{2}$, so we can choose a constant $b$, such that $\psi(R)< b < 1.$ Let $B_a(o)$ be the ball centered at $o$ with radius $a $ in $\mathbb{R}^{4}$. Define $f: \Sigma\cap B_a(o) \To \mathbb{R}$ by
\[
f=\frac{(a^2-r^2)^2\abs{dJ}^2}{(b-\psi \circ J)^2}.
\]
Then by a similar proof of  \cite[Theorem 2]{CheJosQiu12}, we conclude that
\begin{equation*}
\abs{dJ}^2 \leq \max\left\{  \frac{64 r^2}{C_{4}^2(a^2-r^2)^2(b-\psi \circ J)^2}, \quad \frac{32 r^2}{C_4(a^2-r^2)^2} +\frac{8(r+2)}{C_4(a^2-r^2)} \right\},
\end{equation*}
where $C_4$ is a positive constant.
From this we can obtain the upper bound of $f$. Hence at every point of $\Sigma\cap B_{\frac{a}{2}}(o)$, we have
\begin{equation}\label{2.5}
\abs{dJ}^2\leq \frac{C_5}{a^2}.
\end{equation}
Here $C_5$ is a positive constant depending only on $R$.  For any fixed $x$ and letting $a \rightarrow \infty$ in (\ref{2.5}), we then derive that $dJ=0$, namely, $J$ must be constant. It follows that $\mathbf{H} \equiv 0$. Then by Proposition 3.2 in \cite{HanLi09}, $\mathbf{B} \equiv 0$. Hence $\Sigma $ is a plane.
\end{proof}

\begin{rem}
\begin{itemize}
    \item[(1)] Let $\alpha$ be the K\"ahler angle of the translator, \autoref{thm:soliton} implies that  the complete symplectic translating soliton surface with flat  normal bundle and $\cos\alpha$ has a positvie lower bound has to be a plane. Han-Sun  \cite{HanSun10} showed that if $\cos\alpha$ has a positive lower bound, then complete symplectic translating soliton surfaces with bounded second fundamental form and nonpositive normal curvature must be a plane, which indicated that when the normal bundle is flat, such translator is a plane (see \cite[Main Theorem 1]{HanSun10}). In this case, we could remove the condition on the boundedness of the second fundamental form. 
    \item[(2)] The restriction on the image of the complex phase map in \autoref{thm:soliton} is necessary. For example, the ``grim reaper" $(x, y, -\ln\cos x,0), \abs{x}<\pi/2, y\in\mathbb{R}$ is a translating soliton to the symplectic MCF which translates in the direction of the constant vector $(0,0,1,0)$, and $J=(\cos x,0,-\sin x), \abs{x}<\pi/2$ can not contained in any regular ball of $\mathbb{S}^2$. One can check that $\abs{\mathbf{B}}^2=\abs{\mathbf{H}}^2=\abs{\dif J}^2=\cos^2x$. In particular, both the tangent bundle and the normal bundle are flat. 
\end{itemize}

\end{rem}

Now we are at a position to give a proof of \autoref{thm:sing}.

\begin{proof}[Proof of \autoref{thm:sing}]

Firstly for the compact subset $K_1\coloneqq J\left(\Sigma_0\right)\subset \mathbb{V}$, there is a positive and strictly convex smooth function $\rho$ on  $K_1$ (cf. \cite{JosXinYan12}). Choose a domain $U\Subset \mathbb{V}$ such that $K_1\subset U$ and $\rho$ is a strictly convex on $\bar U$. Put $c\coloneqq\max_{K_1}\rho$ and consider the function $u\coloneqq \rho\circ J$. Then $u$ is well defined in $\Sigma\times[0,t_0]$ for small $t_0>0$. 
According to \autoref{thm: harmonic-heat}, along the mean curvature flow,  the complex map satisfies
\begin{align*}
\dfrac{\partial J}{\partial t}=\tau\left(J\right).
\end{align*}
Thus,
\begin{align*}
\dfrac{\partial u}{\partial t}-\Delta u\leq0.
\end{align*}
As a consequence, $u\leq c$ in $\Sigma\times[0,t_0]$.

Let $U_{\varepsilon}$ be a $\varepsilon$-neighborhood of $U$. We claim that 
\begin{claim}
There is a $\varepsilon_0>0$ depending only on $U$ such that for $0<\varepsilon<\varepsilon_0$, we have
\begin{align*}
\set{\mathbf{x}\in \bar U:\rho\left(\mathbf{x}\right)\leq c}=\set{\mathbf{x}\in \overline{ U_{\varepsilon}}:\rho\left(\mathbf{x}\right)\leq c}.
\end{align*}
\end{claim}
Indeed, for every $\mathbf{x}\in\partial U$, define $r\left(\mathbf{x}\right)<\pi/2$ to be the largest number of $r$ such that $\mathrm{B}^{\mathbb{S}^2}_{r}\left(\mathbf{x}\right)\subset\mathbb{V}$. Since $\rho$ is strictly convex on $\bar U$, we can take $0<\varepsilon_0\leq\min_{\partial U}\set{r}$ such that $\rho$ is strictly convex on $\overline{ U_{\varepsilon_0}}$.  If for some $\left(\mathbf{y}\right)\in \overline{ U_{\varepsilon}}\setminus\bar U$, we also have $\rho\left(\mathbf{y}\right)\leq c$. Applying the  maximum principle, we can choose $\mathbf{x}\in\partial U$ with $\rho\left(\mathbf{x}\right)=c$. Let $\gamma:[0,1]\To U_{\epsilon_0}$ be the shortest geodesic from $\mathbf{x}$ to $\mathbf{y}$.  Since $\rho$ is strictly convex, we know that $f\coloneqq\rho\circ\gamma$ is also a strictly convex function on $[0,1]$. Moreover $f'(0)>0$ which is impossible by the maximum principle. Thus the Claim holds.

 Let $\tau\in(0,T]$ be the maximum time such that 
\begin{align*}
    J\left(\Sigma_{t}\right)\subset\set{\mathbf{x}\in \bar U:\rho\left(\mathbf{x}\right)\leq c},\quad\forall 0\leq t<\tau.
\end{align*}
If $\tau<T$, then $J\left(\Sigma_{\tau}\right)\subset\set{\mathbf{x}\in \bar U:\rho\left(\mathbf{x}\right)\leq c}$. Applying the maximum principle  and the above claim, we can extend $\tau$ to some $\tau'>\tau$ which is a contradiction. Thus we obtain that 
\begin{align*}
J\left(\Sigma_{t}\right)\subset\set{\mathbf{x}\in \bar U:\rho\left(\mathbf{x}\right)\leq c},\quad\forall 0\leq t<T.
\end{align*}
Denote $K\coloneqq \set{\mathbf{x}\in \bar U:\rho\left(\mathbf{x}\right)\leq c}$. Clearly, $K_1\subset K$ and $K$ is compact.

Suppose that the mean curvature flow has a Type  \Rmn{1} singularity at $T$. 
Assume 
\begin{align*}
\varepsilon_k=\abs{\mathbf{B}}\left(x_k,t_k\right)=\max_{t\leq t_k}\abs{\mathbf{B}},
\end{align*}
and $x_k\to p\in\Sigma, t_k\to T, F(x_k,t_k)\to q\in M $ as $k\to\infty$.   Set 
\begin{align*}
F_k\left(x,t\right)\coloneqq\varepsilon_k\left(F\left(x,\varepsilon_k^{-2}t+t_k\right)-q\right).
\end{align*}
Denote by $\Sigma_t^{k}\coloneqq F_k\left(\cdot,t\right)\left(\Sigma\right)$, then
\[
g^k_{ij}= \varepsilon^2_{k} g_{ij}, \quad (g^k)^{ij} = \varepsilon^{-2}_k g^{ij}.
\]
Direct computation gives us 
\begin{align*}
\dfrac{\partial F_k}{ \partial t} = \varepsilon^{-1}_k \dfrac{\partial F}{ \partial t} = \mathbf{H}_k , \quad \Delta_{g^k} F_k =  \varepsilon^{-1}_k \Delta F,\quad 
\abs{\mathbf{B}_k}^2 = \varepsilon^{-2}_k \abs{\mathbf{B}}^2.
\end{align*}
Thus
\begin{align*}
    \abs{\mathbf{B}_k}\leq1,\quad\abs{\mathbf{\mathbf{B}}_k\left(x_k,0\right)}=1.
\end{align*}
 Therefore there exists a subsequence of $F_k$, we still denote it by $F_k$, such that $F_k \to F_\infty$ as $k \to \infty$ in any ball $B_R(0) \subset \mathbb{R}^4$, and $F_\infty$ satisfies 
 \begin{align}\label{eqn-limit}
 \frac{\partial F_\infty}{\partial t} = \mathbf{H}_\infty, \quad \abs{\mathbf{B}_{\infty}}\leq 1,   \quad \text{and}\quad \abs{\mathbf{B}_\infty(p, 0)}=1.
 \end{align}
Using the   blow up analysis of the mean curvature flow  (cf. \cite{AreSun13, CheLi01}),  the blow up limit $\tilde{\Sigma}$ is a self-shrinker and complete. By the monotonicity formula, it is easy to see that $\tilde{\Sigma}$ has polynomial volume growth (see \cite[Lemma 2.9 and Corollary 2.13]{ColMin12}),  then by \cite[Theorem 4.1]{CheZho13}, $\tilde{\Sigma}$ is proper.

As the complex phase map is rescaling invariant, we conculde that $J\left(\Sigma_t^k\right)\subset K$ since $J(\Sigma_t) \subset K$. Then by an elementary topology argument, we get $J(\tilde{\Sigma})\subset K$.  Hence  we obtain a complete proper self-shrinker in $\mathbb{R}^4$ with the image of the complex phase map contained in $K$. Applying  \autoref{thm:rigid}, $\tilde{\Sigma}$ must be a plane, which contradicts with $|\mathbf{B}_\infty(p, 0)| > 0$ in  (\ref{eqn-limit}). Thus we complete the proof.
\end{proof}

\begin{rem}

By a similar method of the proof in \autoref{thm:sing}, we can also demonstrate that the following holds:

\begin{quote}
    Let $\Sigma_0$ be a closed hypersurface immersed in Eucildean space $\mathbb{R}^{n+1}$, and  $\Sigma_t \subset \mathbb{R}^{n+1} (t\in [0, T) $ for some $T>0)$ a family of hypersurfaces given by the mean curvature flow. Suppose that the  image of $\Sigma_0$ under the Gauss map is contained in $\mathbb{S}^n\backslash \overline{\mathbb{S}}^{n-1}_{+}$, then the mean curvature flow does not develop any Type \Rmn{1} singularity. 
\end{quote}
 
\end{rem}

\vskip12pt

The following example shows that the restriction on the image of the complex phase map is sharp in \autoref{thm:sing}.

\begin{eg}\label{eg:calabi-tori}
Let $\gamma:\mathbb{S}^1\To\mathbb{C}$ be an immersed curve with $\mathbf{0}\notin\gamma$ and define
\begin{align*}
F:\mathbb{T}^2\To\mathbb{C}^2,\quad(x,y)\to\left(\gamma(x)\cos(y), \gamma(x)\sin(y)\right).
\end{align*}
Then $F$ is a Lagrangian immersion. Since the initial surface is closed, we know that the mean curvature flow always blows up at a finite time (cf. \cite[Proposition 3.10]{Smo12}). 

Denote $\Sigma\coloneqq F\left(\mathbb{T}^2\right)$. Fisrtly, we have the following fact: the Maslov index  of the Lagrangian immersion $\Sigma\To\mathbb{C}^2$ is zero (i.e., $\Sigma$ is of zero-Maslov class) iff 
\begin{align*}
    \mathrm{Ind}_{\gamma}\left(\mathbf{0}\right)=\dfrac{1}{2\pi}\int_{\gamma}\kappa,
\end{align*}
where $\kappa$ is the curvature of the curve $
\gamma$ in $\mathbb{C}$. We will give some more details as follows. The induced metric $g$ on $\Sigma$ is 
\begin{align*}
g=\abs{\gamma'(x)}^2\dif x^2+\abs{\gamma(x)}^2\dif y^2.
\end{align*}
Choose an orientation on $\Sigma$ as following:
\begin{align*}
\dif\mu_{\Sigma}=\abs{\gamma(x)}\abs{\gamma'(x)}\dif x\wedge\dif y.
\end{align*}
Then the pullback of the holomorphic symplectic $2$-form $\Omega$ is
\begin{align}\label{eq:maslov0}
\Omega\vert_{\Sigma}=\gamma(x)\gamma'(x)\dif x\wedge\dif y=\dfrac{\gamma(x)\gamma'(x)}{\abs{\gamma(x)}\abs{\gamma'(x)}}\dif\mu_{\Sigma}=\dfrac{\left(\gamma^2(x)\right)^{'}}{\abs{\left(\gamma^2(x)\right)^{'}}}\dif\mu_{\Sigma}.
\end{align}
Thus, $\Sigma$ is of Maslov class iff the winding number of the curve $x\mapsto \gamma(x)\gamma'(x)$ around  $\mathbf{0}$ is zero iff the degree of the map $x\mapsto \frac{\left(\gamma^2(x)\right)^{'}}{\abs{\left(\gamma^2(x)\right)^{'}}}$ is zero.

For the immersed curve $\gamma$ in $\mathbb{C}$, the unit tangent vector is
\begin{align*}
\vec{e}=\dfrac{\gamma'}{\abs{\gamma'}},
\end{align*}
and the unit outward normal vector is
\begin{align*}
\vec{n}=-\sqrt{-1}\vec{e}=\dfrac{-\sqrt{-1}\gamma'}{\abs{\gamma'}}.
\end{align*}
Thus, the curvature vector is
\begin{align*}
\vec{\kappa}=\bar\nabla_{\vec{e}}\vec{e}=\kappa\vec{n}
\end{align*}
where
\begin{align*}
\kappa=\dfrac{\sqrt{-1}\left(\gamma''\bar\gamma'-\bar\gamma''\gamma'\right)}{2\abs{\gamma'}^2}=- {\rm Im}\left(\ln\gamma'\right)'.
\end{align*}

Notice that the winding number of the curve $\gamma$ around $\mathbf{0}$ is 
\begin{align*}
\mathrm{Ind}_{\gamma}\left(\mathbf{0}\right)=\dfrac{1}{2\pi\sqrt{-1}}\int_{\gamma}\dfrac{\dif\gamma}{\gamma}.
\end{align*}
An immediately consequence is that the degree of the Gauss map $\vec{n}:\mathbb{S}^1\To\mathbb{S}^1$ is
\begin{align*}
\deg{\vec{n}}=-\dfrac{1}{2\pi}\int_{\gamma}\kappa=\dfrac{1}{2\pi\sqrt{-1}}\int_{\gamma}\dfrac{\dif\gamma'}{\gamma'}=\mathrm{Ind}_{\gamma'}\left(\mathbf{0}\right).
\end{align*}
Therefore, according to \eqref{eq:maslov0}, $\Sigma$ is of zero-Maslov class iff 
\begin{align*}
\mathrm{Ind}_{\gamma}\left(\mathbf{0}\right)-\dfrac{1}{2\pi}\int_{\gamma}\kappa\left(\gamma\right)=-\dfrac{1}{2\pi}\int_{\gamma^2}\kappa\left(\gamma^2\right)=0.
\end{align*}

The mean curvature vector $\mathbf{H}$ of $\Sigma$ in $\mathbb{C}^2$ is 
\begin{align*}
\mathbf{H}=&\dfrac{1}{\abs{\gamma(x)}\abs{\gamma'(x)}}\left(\left(\abs{\gamma(x)}\abs{\gamma'(x)}^{-1}\gamma'(x)\right)'\cos(y),\left(\abs{\gamma(x)}\abs{\gamma'(x)}^{-1}\gamma'(x)\right)'\sin(y)\right)\\
&-\dfrac{1}{\abs{\gamma(x)}\abs{\gamma'(x)}}\left(\abs{\gamma(x)}^{-1}\abs{\gamma'(x)}\gamma(x)\cos(y),\abs{\gamma(x)}^{-1}\abs{\gamma'(x)}\gamma(x)\sin(y)\right).
\end{align*}
We compute
\begin{align*}
\dfrac{1}{\abs{\gamma}\abs{\gamma'}}\left(\abs{\gamma}\abs{\gamma'}^{-1}\gamma'\right)'-\dfrac{\gamma}{\abs{\gamma}^2}=&\dfrac{\gamma'}{\abs{\gamma'}^2}\left(\dfrac{\gamma'}{2\gamma}+\dfrac{\bar\gamma'}{2\bar\gamma}+\dfrac{\gamma''}{2\gamma'}-\dfrac{\bar\gamma''}{2\bar\gamma'}\right)-\dfrac{\gamma}{\abs{\gamma}^2}\\
=&\kappa(\gamma)\sqrt{-1}\dfrac{\gamma'}{\abs{\gamma'}}+\dfrac{\gamma'}{\abs{\gamma'}^2}\left(\dfrac{\gamma'}{2\gamma}-\dfrac{\bar\gamma'}{2\bar\gamma}\right)\\
=&\vec{\kappa}(\gamma)-\dfrac{\gamma^{\bot}}{\abs{\gamma}^2}.
\end{align*}
Thus,
\begin{align*}
\mathbf{H}=&\left(\left[\vec{\kappa}(\gamma(x))-\dfrac{\gamma(x)^{\bot}}{\abs{\gamma(x)}^2}\right]\cos(y),\left[\vec{\kappa}(\gamma(x))-\dfrac{\gamma(x)^{\bot}}{\abs{\gamma(x)}^2}\right]\sin(y)\right).
\end{align*}

Therefore the MCF \eqref{eqn-MCF} is reduced to (cf. \cite{Nev07})
\begin{align}\label{eq:csf}
\begin{cases}
\dfrac{\partial\gamma}{\partial t}=\vec{\kappa}(\gamma)-\dfrac{\gamma^{\bot}}{\abs{\gamma}^2},&\mathbb{S}^1\times[0,T);\\
\gamma(\cdot,0)=\gamma_0\left(\cdot\right),&\mathbb{S}^1.
\end{cases}
\end{align}
When $\gamma_0:\mathbb{S}^1\To\mathbb{S}^1, x\mapsto x$, then the  solution of \eqref{eq:csf} is
\begin{align*}
    \gamma(x,t)=2\sqrt{T-t}\gamma_0(x),\quad 0\leq t<T.
\end{align*}
One can check that the Maslov index of $\Sigma_0$ is not zero since 
\begin{align*}
    \mathrm{Ind}_{\gamma_0}\left(\mathbf{0}\right)-\dfrac{1}{2\pi}\int_{\gamma_0}\kappa(\gamma_0)=2.
\end{align*}
In particular, the image of the complex phase map is a great circle. 
Moreover, along the mean curvature flow 
\begin{align*}
    \abs{\mathbf{B}}^2=\dfrac{1}{2(T-t)},\quad\mathbf{H}=-\dfrac{1}{2(T-t)}F,\quad 0\leq t<T.
\end{align*}

\end{eg}

\vspace{2ex}




\end{document}